\documentclass[final]{mcom-l}
\usepackage{enumerate}
\usepackage{amsmath,amssymb,bm}
\usepackage{booktabs}
\usepackage{graphicx}
\usepackage[usenames,dvipsnames]{color}
\usepackage{comment}

\newtheorem{theorem}{Theorem}[section]
\newtheorem{lemma}[theorem]{Lemma}
\newtheorem{corollary}[theorem]{Corollary}

\theoremstyle{definition}
\newtheorem{definition}[theorem]{Definition}

\theoremstyle{remark}

\numberwithin{equation}{section}

\newcommand{\sep}{\vspace{0.5em}}

\newcommand{\dm}[1]{$\displaystyle #1$}

\newcommand{\RR}{{\mathbb{R}}}
\newcommand{\CC}{{\mathbb{C}}}
\newcommand{\h}[1]{\widehat{#1}}
\newcommand{\w}[1]{\widetilde{#1}}
\newcommand{\Zm}{\mathbb{Z}^m}
\newcommand{\IN}{\mathbb{I}_N}
\newcommand{\Nb}{\overline{N}}
\newcommand{\yl}{y_\ell}
\newcommand{\rl}{r_\ell}
\newcommand{\rk}{\rho_k}
\newcommand{\KK}{\w{K}}
\newcommand{\BR}{B_R}
\newcommand{\vphi}{\varphi}
\newcommand{\vphic}{\varphi^c}
\newcommand{\norm}[1]{\Vert{#1}\Vert}
\newcommand{\Ci}{\operatorname{Ci}}
\newcommand{\Si}{\operatorname{Si}}
\newcommand{\sinc}{\operatorname{sinc}}
\newcommand{\erf}{\operatorname{erf}}

\begin{document}

\title[Singular Quadrature Rules and Fast Convolutions]{Singular Quadrature Rules and Fast Convolutions for Fourier Spectral Methods}

\author{Jae-Seok Huh}
\address{Computational Mathematics Group, Computer Science and Mathematics Division, Oak Ridge National Laboratory, Oak Ridge, TN 37831.}
\email{huhj@ornl.gov}

\author{George Fann}
\address{Computational Mathematics Group, Computer Science and Mathematics Division, Oak Ridge National Laboratory, Oak Ridge, TN 37831.}
\email{fanngi@ornl.gov}


\date{June 28, 2012}


\begin{abstract}
We present a generic scheme to construct corrected trapezoidal rules with spectral accuracy
for integral operators with weakly singular kernels in arbitrary dimensions.
We assume that the kernel factorization of the form, $K=\alpha\,\phi+\w{K}$ with smooth $\alpha$ and $\w{K}$, is available
so that the operations on the smooth factors can be performed accurately on the basis of standard Fourier spectral methods.
To achieve high precision results, our approach utilizes the exact evaluation of the Fourier coefficients of the radial singularity $\phi$,
which can be obtained in arbitrary dimensions by the singularity isolation/truncation described in this article.
We provide a complete set of formulae for singularities of the type: $\log(r)$ and $r^{-\nu}$.
Convergence analysis shows that the constructed quadrature rules exhibit almost identical rate of convergence to the trapezoidal rule
applied for non-singular integrands.
Especially, for smooth data, the corrected trapezoidal rules converge super-algebraically.
\end{abstract}

\maketitle

\section{Introduction}
We consider the problem of evaluating the integral,
\begin{equation}
\label{eqn:Ix}
	(If)(x) = \int_D K(x,y) \, f(y) \, dy ,
\end{equation}
where the kernel $K$ may have a point singularity at $x=y$.
We assume that $f$ is compactly supported on $D$ and can be represented satisfactorily by a Fourier series.
When the kernel is also smooth (or sufficiently regular for the purpose),
the usual trapezoidal rule with the integrand $(K\cdot f)$ evaluated on a uniform grid
is a classical rule of thumb for the construction of a spectral scheme involving the integral operator (\ref{eqn:Ix}).
Additionally, many kernels of interest are given by $K(x,y)=K(x-y)$.
Then, the Fourier representation of the kernel and the data enables the fast evaluation of the multiplier operator $I$ via the FFT.
However, the singularities in most kernels of interest render the trapezoidal rule unapplicable,
and we need a corrected one.

We assume that the kernel can be factored to the form,
\begin{equation}
\label{eqn:K}
	K(x,y) = \alpha(x,y) \, \phi(r) + \KK(x,y)
\end{equation}
where $\alpha$ and $\KK$ are smooth functions and $r=\norm{y-x}$ is the Euclidean distance in $\RR^m$.
We assume that the singularity is known and carried entirely by the radial function $\phi$.
For the derivation of the quadrature weights, $\alpha$ and $\KK$ may depend explicitly on the target point $x$,
in which case the weights should be constructed for each $x$.
The main subject of this paper is the construction of corrected trapezoidal rules which exhibit 
the same spectral accuracy as the usual trapezoidal rule applied for smooth kernels.

The constructed quadrature rules can be written as
\begin{equation}
	(If)(x) \sim h \sum_{y_\ell\in\BR} w_\ell f(y_\ell) + h \sum_{y_\ell\ne x} K(x,y_\ell) f(y_\ell) ,
\end{equation}
where the contribution of the singularity is reproduced by the correction weights $w_\ell$ in a neighborhood $\BR$ of $x$.
The presented scheme can also be viewed as a regularization of the kernel such that
the replacement of $\phi$ with the regularized $\w{\phi}$ does not impair the accuracy of the regular trapezoidal rule with a smooth kernel at the given sampling frequency.
In what follows are a few comments regarding the key features of our method and the relationship to other quadrature rules and applications available in the literature.

\begin{enumerate}[(a)]
\item
One of the main advantages of our approach is the high precision of the resulting quadrature rules.
The construction does not involve any other non-trivial quadrature rule or a solve of a linear system
both of which typically limit the achievable order of accuracy.
In our approach, the FFT is the only numerical algorithm used and
the accumulation of the rounding error of the FFT is only $O(\epsilon_{mach}\log N)$.
To achieve the high precision, the Fourier coefficients of the truncated singularity $\phi$ are evaluated up to the machine precision.
One of the major contributions of this paper is the complete recipe we present for the logarithmic and the power-law singularities.
\sep
\item
The second advantage is its applicability for any dimension.
The key functions required for the construction can obtained by recurrence relations presented in main body of the paper.
Except for the $r^{-\nu}$ singularities with a non-integer $\nu$ (which is less significant in applications than the integer cases, we think),
the construction does not require an implementation of a new special function. 
\sep
\item
We present the factored forms of the Helmholtz kernels in arbitrary dimensions so that the corrected quadrature rules can be readily obtained from them.
The Helmholtz kernels are of great importance in applications.
We present numerical examples which demonstrate the advantage of the presented quadrature rules in the construction of high order integral operators
especially for oscillating kernels with high wavenumbers.
\sep
\item
We do not pursue in this study the end-point corrected trapezoidal rules for non-periodic data, which can be found in \cite{Al95,KaRo97,Ro90}.
In principle, an end-point correction is equivalent to an accurate estimation of derivatives at and near the end-points,
which can be done by an over-sampling or by a certain type of data extension.
For the latter, a simple extrapolation results in a terrible oscillation, which limits the achievable order of accuracy.
Hence, most successful \emph{high-order} methods involves a solve of a least-squares problem.
An accurate and stable extension is possible also for the trigonometric basis functions (cf. \cite{Hu09}).
We can represent a smooth non-periodic function by a smooth compactly supported interior and a near boundary part.
In this paper, we focus on the former case and the correction for the near boundary part will be considered in a separate study.
The quadrature rules in \cite{Al95,KaRo97,Ro90} are for one-dimensional singular integrals and
the multidimensional extension is not so straightforward
since a weak singularity in a higher dimension can not always be represented by a product of one-dimensional weak singularities.
Multi-dimensional singular quadrature rules of relatively low orders can be found in \cite{St95}, which are applied for non-oscillatory kernels.
\sep
\item
When $f$ is periodic and the domain of the convolution (\ref{eqn:Ix}) is unbounded,
the kernel can be factored as $K=\vphi\,K+(1-\vphi)K$ by a smooth cut-off function $\vphi$.
The convolution can be recast to
\begin{equation}
	I f = \int K\cdot(\vphi f) + \int \big((1-\vphi)K\big)\cdot f ,
\end{equation}
where the first integral involves the singular kernel and the compactified $(\vphi f)$,
and the second integral the convolution of the periodic $f$ with the regularized kernel $(1-\vphi)K$
where the trapezoidal rule can be applied without any correction.
The Ewald summation (cf. \cite{DaYoPe93}) focuses on the design of the smooth kernel 
so that the Fourier coefficients of the regularized kernel decays rapidly.
Classical applications of the Ewald summation are for point sources, hence, the first integral is merely a summation.
Our focus is more on the accurate evaluation of the first integral for continuous $f$.
\sep
\item
The data $f$ or their Fourier coefficients do not necessarily have to be given on a uniform grid as long as
equivalently fast and accurate algorithms are provided.
For the extension of the quadrature rules to non-uniformly sampled data, one can refer to \cite{DuRo93,GrLe04}.
For sparse data as well as rapidly decreasing kernels, the use of the partial FFT (cf. \cite{BaSw91}) can also be utilized.
\sep
\item
The kernel factorization of the form (\ref{eqn:K}) can be found in numerous articles.
Probably, the most well-known examples are the Nystr\"om methods for boundary integral equations such as \cite[\S3.5]{CoKr98}.
In \cite{BrKu01,YiBiZo06}, combined with partition-of-unity boundary decomposition, Nystr\"om-type schemes are presented for 
boundary integral equations in $\RR^3$.
Spectral schemes in \cite{BrKu01,YiBiZo06} commonly remove the $1/r$-singularity by a polar change of variables,
and evaluate each of the resulting one-dimensional integrals by the trapezoidal rule using resampled data
without building quadrature weights explicitly.
\sep
\item
In order to achieve a proper spectral convergence,
the kernel should be correctly factored so that $\alpha$ and $\KK$ are both smooth.
Regarding this issue, we should keep the fact in mind that $r^2$ is smooth but $r$ is not even in $C^1$.
The following factorization might look attractive since the singularity is isolated independently of the wavenumber.
\begin{equation}
	\frac{e^{ikr}}{4\pi r} = \frac{1}{4\pi r} + \left( \frac{\cos(kr)-1}{4\pi r} + i\frac{\sin(kr)}{4\pi r} \right) .
\end{equation}

However, the above form should be avoided since $(\cos(kr)-1)/r$ is not smooth.
A correct form of factorization is
\begin{equation}
	\frac{e^{ikr}}{4\pi r} = \frac{\cos(kr)}{4\pi r} + i\frac{\sin(kr)}{4\pi r}
\end{equation}
where $\alpha(r)=\cos(kr)/(4\pi)$, $\phi(r)=1/r$, and $\KK(r)=i\sin(kr)/(4\pi r)$.
Readers will soon notice that all the well-behaved functions appearing in this paper have power series expansions containing
only even powers of $r$.
\end{enumerate}

This paper is organized as follows:
in \S2, we begin with the detailed derivation of the quadrature rules,
which will be followed by a complete set of formulae for the evaluation of the Fourier coefficients of the radially-truncated singularities
of the type: $\log(r)$ and $r^{-\nu}$.
Then, we present the factored forms for the Helmholtz kernels in arbitrary dimensions.
In \S3, we discuss the order of accuracy of the quadrature rules related to the regularity of the data.
Finally, the results of numerical experiments will be presented in \S4.

\section{Construction of Quadrature Weights}

\subsection{Notations, grids, and DFTs}
Let $U=[-1,1]^m$ be the computational domain in $\RR^m$.
We consider a uniform grid on $\RR^m$ with the spacing $1/N_j$ in the $j$th dimension.
Thus, for any multi-index $\ell=(\ell_1,\ldots,\ell_m)\in\Zm$,
the corresponding grid point $\yl=(y_{\ell_1},\ldots,y_{\ell_m})$ is given by $y_{\ell_j} = \ell_j/N_j$.
We assume that the physical domain is given as an affine image of $U$ (without translation for the sake of simplicity),that is, a parallelotope in $\RR^m$.
We denote by $\chi$ this non-degenerate linear mapping $U\mapsto \chi U$, then the related metric tensor is constant and is given by $\chi^T\chi$.
We use the symbol $r$ for the Euclidean distance in the physical domain.
That is,
\begin{equation}
	r(x,y) \equiv \norm{\chi(y-x)} .
\end{equation}
With the target $x\in U$ fixed and $r$ viewed as a function of $y\in U$,
we view (\ref{eqn:Ix}) as an integral on $U$ with the Jacobian determinant $|\chi|\equiv\sqrt{\det(\chi^T\chi)}$ multiplied afterward.
The function $r$ is not smooth at the point of singularity due to the square root but $r^2$ is a smooth quadratic polynomial.
Note also that the singularity $\phi$ is not radial in $U$ (but is radial in $\chi U$).

Let $\Nb\equiv\prod_{j=1}^mN_j$ and define $\IN\subset\Zm$ by
\begin{equation}
	\IN \equiv \big\{ (\ell_1,\ldots,\ell_m) \,|\, -N_j \le \ell_j \le N_j-1 \big\} .
\end{equation}
Let $F$ be a periodic function on $U$.
Application of the trapezoidal rule to the (periodic) Fourier transform
\begin{equation}
\label{eqn:ft}
	\h{F}(k) =  \frac{1}{2^m} \int_{\RR^m} F(y)\,e^{-i\pi k\cdot y} \,dy
\end{equation}
results in the formula for the DFT
\begin{equation}
\label{eqn:dft}
	\h{F}_k =  \frac{1}{2^m\Nb} \sum_{\ell\in\IN} F(\yl) \,e^{-i\pi k\cdot\yl} .
\end{equation}
The truncation applied to the inverse Fourier transform (of delta functions on $\Zm$)
\begin{equation}
\label{eqn:ift}
	F(y) = \sum_{k\in\Zm} \h{F}(k)\,e^{i\pi k\cdot y} 
\end{equation}
defines the inverse DFT by
\begin{equation}
\label{eqn:idft}
	F_\ell = \sum_{k\in\IN} \h{F}_k\,e^{i\pi k\cdot\yl} .
\end{equation}
Note that the Fourier coefficients $\h{F}_k$ approximated by the DFT contains error due to the aliasing.
Suppose we are given exact samples, i.e. $F_\ell=F(\yl)$.
The interpolation $\w{F}$ of the samples defined by
\begin{equation}
\label{eqn:intp}
	\w{F}(y) = \sum_{k\in\IN} \h{F}_k\,e^{i\pi k\cdot y} \quad\text{where $\h{F}_k$ is the DFT of $F_\ell$}
\end{equation}
involves two sources of error:
(1) the aliasing contained in $\h{F}_k$ and (2) the truncation error of the inverse DFT.
The decay characteristics of the exact Fourier coefficients $\h{F}(k)$ takes a central role in the estimation of  the error.

\subsection{The support of the data}
We assume that the data (or source) $f$ in (\ref{eqn:Ix}) viewed as a function on the computational domain is supported on $[0,1]^m$.
Then, $f$ can be extended on $U$ by padding zeroes.
The rationale behind this is quite obvious;
since we take $x+U$ as the support of the kernel $K(x,\cdot)$ (or equivalently, the domain of integral) for the target $x$,
if $f$ is not extended, the convolution will include the effect of the \emph{fictitious} portion of the periodized source.
Similarly, if $x$ is not in $[0,1]^m$, the result $(If)(x)$ will include the effect of the fictitious source,
hence, the portion on $U\setminus[0,1]^m$ of $(If)(x)$ obtained on $U$ should be discarded.

For $x\not\in[0,1]^m$, the integrand becomes smooth, hence, the usual trapezoidal rule serves our purpose well without any correction.
The resulting summation can be accelerated by any fast multipole (or an equivalent) method designed for the kernel.
In this paper, we focus ourselves on the near-field solution for $x\in[0,1]^m$.

On this setting, the translation $f(y-x)$ by $x\in[0,1]^m$ results in the translated source vanishing on the boundary of $U$ and satisfying all the assumptions.
Therefore, in this section, without loss of generality, we assume that the target point $x$ is at the origin.
Every function is viewed as a function of the source point $y$ only and the symbol $x$ will be omitted.
Thus, we recast (\ref{eqn:Ix}) to the integral,
\begin{equation}
\label{eqn:II}
	I(f) = |\chi| \int_{U} K(y) \, f(y) \, dy .
\end{equation}

\subsection{Localization of the singularity}
Let $r(y)\equiv\norm{\chi y}$.
With a slight abuse of notation, we denote by $\BR$ an ellipsoid in $U$ given by $\BR\equiv\{\, r(y) \le R \,\}$ for some $R>0$,
where $R$ is chosen such that the image $\chi\,\BR$ (a ball of radius $R$ in the physical domain) is contained in $\chi\,U$.
For the accuracy, the best choice of $R$ is $\max_{y\in\partial U} r(y)$ so that $\BR$ contains as many grid points as possible.

Let $\vphi$ be any smooth even function on $\RR$ such that
(1) $\vphi(0)=1$, (2) $\vphi=0$ on $\RR\setminus(-R,R)$, and (3) $\vphic=(1-\vphi)$ vanishes smoothly at the origin.
Then, $(\phi(r)\,\vphic(r))$ is also a smooth function and vanishes at the origin.
For all the numerical experiments in this paper, we utilized $\vphi(r)=\vphi_1(r/R)$ where
$\vphi_1$ is the sigmoidal function,
\begin{equation}
	\vphi_1(t) = \begin{cases} e^{-e^{-2/|t|}/(1-|t|)^2} & |t| < 1 \\ 0 & |t| \ge 1 \end{cases} .
\end{equation}

Then, we isolate the singularity by using the identity
\begin{equation}
	\phi(r) = \phi(r)\,\vphi(r) + \phi(r)\,\vphic(r)
\end{equation}
and rewrite (\ref{eqn:II}) as
\begin{equation}
	I(f) = |\chi| \int_{B_R} \phi(r)\,F(y) \,dy + |\chi| \int_{U} G(y) \,dy
\end{equation}
where $F$ and $G$ are regular periodic functions given by
\begin{align}
	F(y) &= \vphi(r)\,\alpha(y)\,f(y) \\
	G(y) &= \vphic(r)\,\phi(r)\,\alpha(y)\,f(y) + \KK(y)\,f(y) .
\end{align}

The second integral without the singularity can be treated well by the usual trapezoidal rule; that is,
\begin{equation}
\label{eqn:quad:G}
	\int_{U} G(y) \,dy \sim \frac{|\chi|}{\Nb} \sum_{\ell\in\IN} \left( \vphic(\rl)\,\phi(\rl)\,\alpha(\yl) + \KK(\yl) \right) f(\yl) .
\end{equation}
For the first integral with the singularity, we utilize the interpolation (\ref{eqn:intp}) to obtain
\begin{align}
	|\chi| \int_{\BR} \phi(r) \, F(y) \, dy
	&\sim |\chi| \sum_{k\in\IN} \h{F}_k \int_{\BR} \phi(r) \, e^{i\pi k\cdot y} \, dy \\
	&= \frac{|\chi|}{\Nb} \sum_{\ell\in\IN} F(\yl) \sum_{k\in\IN} \left( \frac{1}{2^m} \int_{\BR} \phi(r) \, e^{i\pi k\cdot y} \, dy \right) \\
	&\equiv \frac{|\chi|}{\Nb} \sum_{\ell\in\IN} F(\yl) \sum_{k\in\IN} \h{\phi}(k) \\
	&\equiv \frac{|\chi|}{\Nb} \sum_{\ell\in\IN} F(\yl) \, \w{\phi}_\ell ,
	\label{eqn:quad:F}
\end{align}
where we defined
\begin{equation}
	\h{\phi}(k) \equiv \frac{1}{2^m} \int_{\BR} \phi(r) \, e^{-i\pi k\cdot y} \, dy
	\quad\text{and}\quad
	\w{\phi}_\ell \equiv \sum_{k\in\IN} \h{\phi}(k) .
\end{equation}
We can change the sign in the exponential function arbitrarily since $r(-y)=r(y)$.
Notice that $\h{\phi}(k)$ are the \emph{exact} Fourier coefficients of $\phi(r)$
which is truncated to zero on the exterior of $\BR$ and viewed as a periodic function on $U$.
Or equivalently, we can consider $\h{\phi}$ as the (non-periodic) Fourier transform (scaled by $2^{-m}$) of $\phi$ truncated on $\RR^m\setminus\BR$.
By definition, $\w{\phi}_\ell$ are the inverse DFT of the finite samples $\h{\phi}(k)$ in the frequency domain.
Consider the interpolation
\begin{equation}
	\w{\phi}(y) = \sum_{k\in\IN} \h{\phi}(k)\,e^{i\pi k\cdot y} .
\end{equation}
Then, $\w{\phi}_\ell=\w{\phi}(\yl)$ and (\ref{eqn:quad:F}) is simply the trapezoidal rule applied to the product of the two functions, $\w{\phi}$ and $F$.
Thus, the procedure is equivalent to the regularization of the singular $\phi$ such that the regularized kernel $\w{\phi}$ results in the exact integral
by the trapezoidal rule if $F$ can be exactly represented on the given grid. 

Care should be taken not to confuse $\h{\phi}$ with the Fourier transform of the original \emph{non-truncated} $\phi$.
To be more precise, we should use a notation like $\h{\phi}_R$, but we choose the notational simplicity.
Typically, $\phi$ is a slowly decaying function with the point singularity at the origin, hence, without the truncation,
its Fourier transform possesses the same nature in the frequency domain.
By truncating in the space, we regularize the Fourier transform to a smooth (but still slowly decaying) $\h{\phi}_R$.
By truncating in the frequency domain, that is, by sampling only up to the given sampling frequency, we obtain $\w{\phi}$ regularized in the space.

As one can notice from the above derivation, our construction requires the exact evaluation of $\h{\phi}$.
In \S\ref{sec:phi}, we present a detailed discussion on the nature of $\h{\phi}$
as well as the explicit formulae for the logarithmic and the power-law singularities.

\subsection{The corrected trapezoidal rule}\label{sec:therule}
Merging (\ref{eqn:quad:G}) and (\ref{eqn:quad:F}), the quadrature rule can be written as
\begin{equation}
\label{eqn:quad:orig}
	I(f) \sim \frac{|\chi|}{\Nb} \sum_{\ell\in\IN}
	\left( \alpha(\yl) \left( \vphi(\rl)\,\w{\phi}_\ell + \vphic(\rl)\,\phi(\rl) \right) + \KK(\yl) \right)
	f(\yl) .
\end{equation}
For $\ell=0$, since $\vphi(0)=1$ and $\vphic(0)=0$, the weight becomes
\begin{equation}
	\alpha(0) \, \w{\phi}_0 + \KK(0) .
\end{equation}
For $\ell\ne0$, we can rewrite
\begin{equation}
	\vphi(\rl)\,\w{\phi}_\ell + \vphic(\rl)\,\phi(\rl) = \left( \w{\phi}_\ell - \phi(\rl) \right) \vphi(\rl) + \phi(\rl) ,
\end{equation}
Thus, we present the quadrature rule by the sum of the trapezoidal rule and the correction rule; that is,
\begin{equation}
\label{eqn:quad:sum}
	I(f) \sim \frac{|\chi|}{\Nb} \sum_{\substack{\ell\in\IN\\ \ell\ne0}} K(\yl) \, f(\yl)
	+ \frac{|\chi|}{\Nb} \sum_{\substack{\ell\in\IN\\ \rl<R}} w_\ell \, f(\yl)
\end{equation}
where the correction weights $w_\ell$ are given by
\begin{equation}
\label{eqn:quad:wgt}
	w_\ell = 
	\begin{cases}
		\alpha(0)\,\w{\phi}_0 + \KK(0) & \ell = 0 \\
		\alpha(\yl) \left( \w{\phi}_\ell - \phi(\rl) \right) \vphi(\rl) & \ell\ne0
	\end{cases} .
\end{equation}

The separated representation (\ref{eqn:quad:sum},\ref{eqn:quad:wgt}) is the final form of the quadrature rule we present in this paper.
The quadrature weights (\ref{eqn:quad:wgt}) is quite self-explanatory;
at the point of singularity, $\phi$ has been represented by the equivalent finite weight $\w{\phi}_0$
which accompanies the balancing neighbors $(\w{\phi}_\ell - \phi(\rl))\vphi(\rl)$ to achieve the desired spectral accuracy.
The advantage of using the separated form (\ref{eqn:quad:sum},\ref{eqn:quad:wgt}) over the primitive one (\ref{eqn:quad:orig}) is obvious;
oftentimes, the smooth remainder term $\KK$ has a quite complicated form.
Since the separated form requires only the limiting value of $\KK$ at the origin, the implementation can be simpler and more readable.

There must be a few careful readers concerning the numerical soundness of the expression $(\w{\phi}_\ell - \phi(\rl))$.
With finite precision arithmetics, it sometimes comes to a catastrophic end to take the difference of two potentially large values;
the cancellation error can corrupt the result.
However, on uniform grids where grid points are not clustered, the issue is not so significant.
Moreover, for spectral methods like ours, when applied to quite smooth data,
the solution often converges to the desired precision before the grid spacing becomes small enough to bring such an issue to the surface.
In all the numerical experiments we conducted including all the examples in this paper,
we utilized the separated representation but have never experience a trouble.
Actually, the classical Nystr\"om methods frequently use the expression $\KK=K-\alpha\,\phi$ to avoid the explicit evaluation of a complicated $\KK$
(see, for example, \cite[\S3.5]{CoKr98}).
For rare cases when the issue becomes significant, one can utilize the original form (\ref{eqn:quad:orig}) for grid points close to the singularity.

We can also bring the convolution (\ref{eqn:quad:sum}) to the frequency domain.
In this sense, our construction is equivalent to the representation of the Fourier transform of $K$ truncated on $\RR^m\setminus U$ by
\begin{equation}
\label{eqn:K:ft}
	\h{K} = \h{\alpha}*\h{\phi}*\h{\vphi} + \left( \h{\alpha}*\h{\phi}*\h{\vphic} + \h{\w{K}} \right)
\end{equation}
where the terms in the parentheses decays rapidly.
The key idea is that the application of the radial cut-off function $\vphi$ enables us to truncate $\phi$ outside of $\BR$ and
we can evaluate the Fourier transform (which is now the one-dimensional Hankel transform) of the truncated $\phi$ exactly.
Numerically, we evaluate $\h{K}$ by the DFT of $\{w_0,w_\ell+K(\yl)\}$, where $w_\ell$ are obtained by the inverse DFT of $\h{\phi}(k)$.

Now, the true meaning of the \emph{grid} we used so far has become clearer.
It is not necessary for the data $f$ to be given on that (or on any) grid.
Regarding $f$, the only information we need is the required minimum sampling frequency.
The grid we have used so far is in principle to perform the DFT of $\w{K}$ and the convolutions (by multiplications in the spatial domain) in (\ref{eqn:K:ft}),
hence, we call it more specifically the construction grid.

\begin{enumerate}[(A)]
\item A recommended sampling frequency for the construction (which is greater or equal to the given sampling frequency of the data)
	is such that the interpolation errors of the smooth functions $\alpha$, $\vphi$, and $\w{K}$ are comparable to that of the data.
	Otherwise, the quadrature weights will still work but the error in the weights will be the dominating factor (cf. \S\ref{sec:conv:II} and \S\ref{sec:conv:II}).
\sep
\item For the inverse DFT to obtain $\w{\phi}$, we do not need to use the entire construction grid.
	Only a subset of the grid containing $\BR$ is sufficient.
	The use of the subset can be very useful if the kernel exhibits exponential decay
	like the Helmholtz kernel with a complex wavenumber $k$ with $\operatorname{Im}(k)>1$.
	Such an exponentially decaying kernel requires a very high sampling frequency for the construction (cf. \S\ref{sec:conv:II}).
	However, due to the rapid decay, the kernel outside of a certain small $\BR$ is practically zero.
	Hence, we can compute $(w_\ell + \w{K}(\yl))$ only within the small subset of the conceptually huge construction grid.
	The DFT on the entire grid can be performed by padding zeros and the unnecessary high frequency terms can be discarded,
	both of which can be done in a single efficient procedure without using huge temporary memory by (so called) the partial FFT.
\end{enumerate}

\subsection{The Fourier transform $\bm{\h{\phi}}$} \label{sec:phi}
The construction of the correction weights requires the exact evaluation of the Fourier transform $\h{\phi}$ of the truncated singularity.
We begin with the definition of $\h{\phi}$,
\begin{equation}
	\h{\phi}(k) = \frac{1}{2^m} \int_{\BR} \phi(r) \, e^{-i\pi k\cdot y} \,dy .
\end{equation}
Note that $\phi$ is not a radial function in the computational domain $U$ and $\BR$ is not a ball in $U$.
First, we perform the change of variables $\eta=R^{-1} \chi y$ back to a scaled physical domain.
Utilizing $k\cdot y=k'\cdot\eta$ with $k'=R \chi^{-T} k$, we obtain
\begin{equation}
	\h{\phi}(k) = \frac{R^m}{2^m|\chi|} \int_{\norm{\eta}\le1} \phi\left(R\norm{\eta}\right) \, e^{-i\pi k'\cdot\eta} \, d\eta .
\end{equation}
The function $\phi(R\norm{\cdot})$ is radial and supported on the unit ball.
Hence, the integral on the right side is a radial function of $k'$ and is given by the following Hankel transform on one-dimension.
\begin{equation}
\label{eqn:phi:J}
	\h{\phi}(k)  = \frac{R^m}{2^m|\chi|} \frac{(2\pi)^{m/2}}{\rk^m}
	\int_0^{\rk} \phi\left( \frac{R}{\rk} t \right) \, t^{m/2} \, J_{(m-2)/2} (t) \,dt
\end{equation}
where $J_\mu$ is the Bessel function of the first kind and
\begin{equation}
	\rk = \pi R \norm{\chi^{-T} k} .
\end{equation}

Then, we can rewrite ($\ref{eqn:phi:J}$) as
\begin{equation}
\label{eqn:phi:A}
	\h{\phi}(\rho) = \frac{\sqrt{\pi}^m}{2^m|\chi|\,\Gamma(m/2+1)} \frac{R^m}{\rho^m} \int_0^\rho \phi\left(\frac{R}{\rho}t\right) \, m \, t^{m-1} \, A_{m}(t) \,dt ,
\end{equation}
where the Bessel function of degree $(m-2)/2$ has been replaced with a better behaved function $A_m$ defined by
\begin{equation}
	A_{m}(t) \equiv \Gamma(m/2) \frac{J_{(m-2)/2}(t)}{(t/2)^{(m-2)/2}} .
\end{equation}
Then, $A_m(0)=1$ and $|A_m|\le1$.
The series representation of $A_m$ can be obtained from that of $J_{(m-2)/2}$.
\begin{equation}
	A_{m}(t) = \sum_{\ell=0}^\infty \frac{(-1)^\ell \Gamma(m/2)}{\ell!\,\Gamma(\ell+m/2)} \left(\frac{t}{2}\right)^{2\ell}
	.
\end{equation}
Thus, $A_m$ is an entire function on $\CC$ and is even on $\RR$.
The asymptotic behavior of $A_m$, which can be obtained from that of $J_{(m-2)/2}$, governs the  decay characteristics of $\h{\phi}$.
For $t \gg m^2$,
\begin{equation}
\label{eqn:Am:asy}
	A_{m}(t) \sim \frac{\Gamma(m/2)}{\sqrt{\pi}} \frac{\cos(t-(m-1)\pi/4)}{(t/2)^{(m-1)/2}} .
\end{equation}

Note that $A_m$ includes a Bessel function with an integer index for an even $m$ and with a half-integer index for an odd $m$,
the latter of which can be expressed by a finite series of trigonometric functions.
Although the series representation may look complicated, $A_m$ actually consists of familiar functions.
For the first 4 indices,
\begin{align}
	A_1(t) &= \cos(t) \\
	A_2(t) &= J_0(t) \\
	A_3(t) &= \sin(t)/t \equiv \sinc(t) \\
	A_4(t) &= 2 J_1(t)/t
	.
\end{align}
For indices greater than 4, closed-forms can be obtained by utilizing the three-term recurrence relation,
\begin{equation}
\label{eqn:Am:rec}
	A_{m+4}(t) = \frac{m(m+2)}{t^2} \left( A_{m+2}(t) - A_{m}(t) \right) ,
\end{equation}
which can be easily verified from the series form.
Thus, $A_m$ involves only $J_0$ and $J_1$ for an even $m$, and only cosine and sine for an odd $m$,
which does not introduce any implementation issue.

Another useful identity is
\begin{equation}
\label{eqn:Am:dif}
	\left( t^m A_{m+2}(t) \right)' = m \, t^{m-1} A_{m}(t) ,
\end{equation}
which results in another expression for (\ref{eqn:phi:A});
suppose $\phi$ is $C^1$ on $(0,R\,]$ and $(\phi(t)\,t^{m-\epsilon})\rightarrow0$ as $t\rightarrow0$ for some $\epsilon>0$.
Applying the integration-by-part, (\ref{eqn:phi:A}) becomes
\begin{equation}
\label{eqn:phi:IBP}
	\h{\phi}(\rho) = \frac{\sqrt{\pi}^m R^m}{2^m|\chi|\,\Gamma(m/2+1)}
	\bigg\{ \phi(R) A_{m+2}(\rho) - \int_0^1 \phi_1(R\,t) \, t^{m-1} \, A_{m+2}(\rho\,t) \,dt \bigg\} ,
\end{equation}
where $\phi_1(t) = t\,\phi'(t)$ and the change of variables $t/\rho\mapsto t$ is applied for the second integral.
Note that $\phi_1(t)=1$ for $\phi(t)=\log(t)$ and $\phi_1(t)=-\nu\,\phi(t)$ for $\phi(t)=t^{-\nu}$,
which provides us with useful recurrence relations for the integral term.
The formula (\ref{eqn:phi:IBP}) also reveals the fact that the asymptotic behavior of $\h{\phi}$ is determined by that of $A_{m+2}$ (not $A_m$),
which is summarized in the following lemma.

\begin{lemma}
\label{lem:phi}
Let the grid parameters, $R$ and $\chi$, be fixed.
Suppose $\phi$ is $C^1$ on $(0,R\,]$ and $(\phi(t)\,t^{m-\epsilon})\rightarrow0$ as $t\rightarrow0$ for some $\epsilon>0$.
Then, for $k\in\Zm$,
\begin{equation}
	| \h{\phi}(k) | = O\left(\norm{k}^{-(m+1)/2}\right) .
\end{equation}
\end{lemma}
\begin{proof}
First, consider the integral in (\ref{eqn:phi:IBP}).
\begin{equation}
	\left|\int_0^1 \phi_1(R\,t)\,t^{m-1}\,A_{m+2}(\rho\,t)\,dt\right| \le
	\frac{\norm{\phi_1(t)\,t^m}_\infty}{R^m} \int_0^1 \frac{ \left| A_{m+2}(\rho\,t) \right| }{t^{1-\epsilon}} \,dt
\end{equation}
where we denote by $\norm{\cdot}_\infty$ the $L^\infty$-norm on $[\,0,R\,]$.
From (\ref{eqn:Am:asy}), there is a constant $C>0$ such that $|A_{m+1}(\rho\,t)| \le C(\rho\,t)^{-(m+1)/2}$.
Define $a$ by $C=(\rho\,a)^{(m+1)/2}$. Then, $a<1$ if $\rho>C^{2/(m+1)}$.
Since $|A_{m+2}|\le 1$,
\begin{equation}
\begin{split}
	\int_0^1 \frac{ \left| A_{m+2}(\rho\,t) \right| }{t^{1-\epsilon}} \,dt
	&\le \int_0^a t^{\epsilon-1} \, dt + C\,\rho^{-(m+1)/2} \int_a^1 t^{\epsilon-(m+1)/2-1} \, dt \\
	&\le \delta \frac{a^\epsilon}{\epsilon} + \frac{C\,\rho^{-(m+1)/2}}{\epsilon-(m+1)/2} \left( 1 - a^{\epsilon-(m+1)/2} \right) \\
	&= \frac{C\,\rho^{-(m+1)/2}}{\epsilon-(m+1)/2} + a^\epsilon \left( \frac{\delta}{\epsilon} - \frac{1}{\epsilon-(m+1)/2} \right) ,
\end{split}
\end{equation}
where $\delta$ is an arbitrary constant $\ge1$.
When $\epsilon\ge(m+1)/2$, $a^\epsilon=O\left(\rho^{-(m+1)/2}\right)$  since $a=O\left(\rho^{-1}\right)$.
When $\epsilon<(m+1)/2$, we can choose $\delta=\epsilon/(\epsilon-(m+1)/2)>1$ so that the second term vanishes.
Thus, $\int_0^1 t^{\epsilon-1} |A_{m+2}(\rho\,t)|\,dt = O\left(\rho^{-(m+1)/2}\right)$.
Combined (\ref{eqn:Am:asy}) for the first term in (\ref{eqn:phi:IBP}), the above result shows that $|\h{\phi}(\rho)|=O\left(\rho^{-(m+1)/2}\right)$.
Since $\rk\le\pi R \norm{\chi^{-T}} \norm{k}$ where $\norm{\chi^{-T}}$ is the operator norm of $\chi^{-T}$, the same relation holds for $\norm{k}$ also.
\end{proof}

In what follows, we present the formulae of $\h{\phi}$ for $\phi(r)=\log(r)$ and  for $\phi(r)=r^{-\nu}$ with $\nu<m$,
which enables us to evaluate them up to the machine precision.
We begin with the logarithmic singularity.

\subsubsection{$\bm{\phi(r)=\log(r)}$}
From the integration-by-parts formula (\ref{eqn:phi:IBP}), we obtain
\begin{equation}
	\h{\phi}^L_m(\rho) = \frac{\sqrt{\pi}^m R^m}{2^m|\chi|\Gamma(m/2+1)}
	\bigg\{ \log(R) A_{m+2}(\rho) - L_{m}(\rho) \bigg\}
\end{equation}
where
\begin{equation}
	L_{m}(\rho) \equiv  \int_0^1 t^{m-1} \, A_{m+2}(\rho\,t) \,dt
	.
\end{equation}
For $m=1$ and 2,
\begin{align}
	L_1(\rho) &= \int_0^1 \frac{\sin(\rho\,t)}{\rho\,t} \,dt = \frac{\Si(\rho)}{\rho} \\
	L_2(\rho) &= \int_0^1 \frac{2 J_1(\rho\,t)}{\rho} \,dt = \frac{2(1-J_0(\rho))}{\rho^2} .
\end{align}
For $m\ge3$, utilize the recurrence relations (\ref{eqn:Am:rec}) to obtain
\begin{equation}
	L_{m+2}(\rho) =  \frac{m(m+2)}{\rho^2} \int_0^1 t^{m-1} \left( A_{m+2}(\rho\,t) - A_{m}(\rho\,t) \right) dt
\end{equation}
where the second term can be integrated by using (\ref{eqn:Am:dif}),
which results in the recurrence relation for $L_m$,
\begin{equation}
	L_{m+2}(\rho) = \frac{(m+2)}{\rho^2} \bigg\{ m \, L_{m}(\rho) - A_{m+2}(\rho) \bigg\} .
\end{equation}
With the formulae for $L_1$ and $L_2$, the above recurrence relation enables us to construct an explicit formula of $\h{\phi}^L_m$ for any $m$.
In order to avoid possible cancellation error of the recurrence formula for small $\rho$, we need to evaluate the series representation,
\begin{equation}
	L_m(\rho) = \sum_{\ell=0}^\infty \frac{(-1)^\ell \Gamma(m/2+1)}{\ell!\,\Gamma(\ell+m/2+1)(m+2\ell)} \bigg(\frac{\rho}{2}\bigg)^{2\ell} ,
\end{equation}
which is also more efficient for small $\rho$.

\subsubsection{$\bm{\phi(r)=r^{-\nu}\,\,(\nu<m)}$}
First, notice that $r^{-\nu}$ with any $\nu<m$ can be factored as $r^{-\nu}=r^{2n}\cdot r^{-\epsilon}$ with a non-negative integer $n$ and $\epsilon\in[m-2,m)$.
Since $r^2$ is smooth, the smooth factor $r^{2n}$ can be included in $\alpha$.
Interestingly, due to the factor $r^{2n}$ in $\alpha$, the correction weight $w_0$ at origin vanishes for $r^{-\nu}$ with $\nu<m-2$.
For such \emph{weaker} singularities, the correction is made by the correction weights (with relatively smaller magnitude) near the location of the singularity.

Therefore, from now on, we presume $\nu\in[m-2,m)$.
It is convenient for the presentation to use a new notation, $\mu\equiv m-\nu$, where $\mu\in(0,2\,]$.
Then, (\ref{eqn:phi:A}) is written in terms of $\mu$ as,
\begin{equation}
	\h{\phi}^{(\mu)}_m(\rho) = \frac{\sqrt{\pi}^m R^\mu}{2^m|\chi|\,\Gamma(m/2+1)} \, M^{(\mu)}_m(\rho)
\end{equation}
where
\begin{equation}
	M^{(\mu)}_m(\rho) \equiv \int_0^1 m \, t^{\mu-1} \, A_{m}(\rho\,t) \,dt .
\end{equation}
The series representation is given by
\begin{equation}
	M^{(\mu)}_m(\rho) = \sum_{\ell=0}^\infty \frac{(-1)^\ell \Gamma(m/2)\,m}{\ell!\,\Gamma(\ell+m/2)(\mu+2\ell)} \bigg(\frac{\rho}{2}\bigg)^{2\ell} ,
\end{equation}
which can be evaluated efficiently for small $\rho$ avoiding the cancellation error of the recurrence relation.

We can derive a recurrence relation for $M_m$ like the one for $L_m$ in the previous section.
By integrating by parts and utilizing $t\,A'_{m+2}(t)=m(A_m(t)-A_{m+2}(t))$,
\begin{equation}
\begin{split}
	\mu\,M^{(\mu)}_{m+2}(\rho)
	&= (m+2)\,A_{m+2}(\rho) - (m+2) \int_0^1 \rho \, t^\mu \, A'_{m+2}(\rho\,t) \, dt \\
	&= (m+2)\,A_{m+2}(\rho) - (m+2)\,M^{(\mu)}_m(\rho) + m\,M^{(\mu)}_{m+2}(\rho) .
\end{split}
\end{equation}
Hence,
\begin{equation}
\label{eqn:Mm:rec:pre}
	(m-\mu)\,M^{(\mu)}_{m+2}(\rho) = (m+2) \big( A_{m+2}(\rho) - \,M^{(\mu)}_m(\rho) \big) .
\end{equation}
And, the above equation implies that $\,M^{(m)}_m(\rho)=A_{m+2}(\rho)$.
Since $\mu\in(0,2\,]$, there are only two cases with $\mu=m$: $\mu=m=1$ and $\mu=m=2$.
Except those two cases, we can utilize the recurrence relation,
\begin{equation}
\label{eqn:Mm:rec}
	M^{(\mu)}_{m+2}(\rho) = \frac{m+2}{m-\mu} \bigg( A_{m+2}(\rho) - \,M^{(\mu)}_m(\rho) \bigg) .
\end{equation}

Among $\mu\in(0,2\,]$, the two integer cases are of prime interest;
(1) $r^{2-m}$ ($\mu=2$) is the principal singularity of the Helmholtz kernel in $\RR^m$
(in even dimensions with $m\ge4$, the Helmholtz kernel contains an additional logarithmic singularity,
which can be treated by $\h{\phi}^L_m$ in the previous section).
Interestingly, this (probably) most important class of singularities has the simplest description;
$M^{(2)}_m$ can be written explicitly without using the recurrence relation.
(2) Singularities with $\mu=1$ arise when the Helmholtz kernel in $\RR^{m+1}$ is acting on $m$-dimensional flat boundary.
For $m=1$, the domain of integral need not be a flat manifold (see an example in \S\ref{sec:bie} for the application of the quadrature rule on curves).
We are studying the extension of our method for higher dimensional general (non-flat) manifolds.

\sep
\begin{enumerate}[(I)]

\item $\bm{\mu=2 \quad(\nu=m-2)}.$
\begin{equation}
	M^{(2)}_m(\rho) = \int_0^1 m \, t \, A_{m}(\rho\,t) \,dt
\end{equation}
The evaluation of $M^{(2)}_1$ is straightforward.
The formula for $M^{(2)}_2$ is the result of (\ref{eqn:Mm:rec:pre}) with $\mu=m$.
For $m>2$, it is not difficult to obtain the formula from the series representation of $A_m$.
\begin{align}
	M^{(2)}_1 &= \frac{\cos(\rho)-1}{\rho^2} - \sinc(\rho) \\
	M^{(2)}_2 &= A_4(\rho) \\
	M^{(2)}_m &= \frac{m(m-2)}{\rho^2} \bigg( 1 - A_{m-2}(\rho) \bigg) \quad\text{for $m>2$}
\end{align}

\item $\bm{\mu=1 \quad(\nu=m-1)}.$
\begin{equation}
	M^{(1)}_m(\rho) = \int_0^1 m \, A_{m}(\rho\,t) \,dt
\end{equation}
Formulae for $M^{(1)}_1$ and $M^{(1)}_3$ requires elementary calculus only.
To the best of our knowledge, there is not a simple representation of $M^{(1)}_2$ by well-known functions.
Hence, we treat the following integral form of $M^{(1)}_2$ as the definition of a special function
(see \cite[p. 480]{AbSt72} for the properties of the integral).
\begin{align}
	M^{(1)}_1(\rho) &= A_3(\rho) \\
	M^{(1)}_2(\rho) &= 2 \int_0^1 J_0(\rho\,t) \,dt = \frac{2}{\rho} \int_0^\rho J_0(t) \,dt \\
	M^{(1)}_3(\rho) &= \frac{3 \Si(\rho)}{\rho}
\end{align}
One may implement his/her own version of $M^{(1)}_2$ from the series representation $J_0$ and the asymptotic expansion of the integral,
or can simply use an implementation of $\int_0^\rho J_0(t)\,dt$ in Algorithm 757 (MISCFUN) of ACM Transactions on Mathematical Software (TOMS) (cf. \cite{Ma96}).
For $m>1$, $M^{(1)}_{m+2}$ can be obtained from the recurrence relation,
\begin{equation}
	M^{(1)}_{m+2}(\rho) = \frac{m+2}{m-1} \bigg( M^{(1)}_{m}(\rho) - A_{m+2}(\rho) \bigg) .
\end{equation}

\item $\bm{\mu\in(0,1)\cup(1,2)}.$\\
The recurrence relation (\ref{eqn:Mm:rec}) can be applied to obtain $M^{(\mu)}_{m+2}$ for any $m>1$.
Hence, we only need to consider two initial cases $m=1, 2$.
\sep
\begin{enumerate}[(a)]
\item $m=1.$
\begin{equation}
\begin{split}
	M^{(\mu)}_1(\rho) &= \int_0^1 t^{\mu-1} \cos(\rho\,t) \,dt = \frac{\Ci(\mu,\rho)}{\rho^\mu} \\
	&= \sum_{\ell=0}^\infty \frac{(-1)^\ell \rho^{2\ell}}{(2\ell+1)!(2\ell+\mu)} .
\end{split}
\end{equation}
The function $\Ci(\mu,\rho)\equiv\int_0^\rho t^{\mu-1}\cos(t)\,dt$ is known as the generalized cosine integral,
which is related to the lower incomplete gamma function with pure imaginary argument,
\begin{equation}
	\Ci(\mu,\rho) = \operatorname{Re}\big((-i)^\mu\,\gamma(\mu,i\rho)\big) .
\end{equation}

\item $m=2.$
\begin{equation}
	M^{(\mu)}_2(\rho) = \int_0^1 t^{\mu-1} J_0(\rho\,t) \,dt
	= \sum_{\ell=0}^\infty \frac{(-1)^\ell (\rho/2)^{2\ell}}{(\ell!)^2(2\ell+\mu)} .
\end{equation}
\end{enumerate}

To the best our knowledge, there is no available/reliable implementation of either $M^{(\mu)}_1$ or $M^{(\mu)}_2$
(or any special function which can be used to compute them) in the public domain.
We can follow the standard implementation procedure of special functions
-- the partial sums of the above power series' for small $\rho$ and the asymptotic expansions for large $\rho$.
For $M^{(\mu)}_1$, the error of the 18-term partial sum is less than $10^{-16}$ on $0\le\rho\le2\pi$.
For $M^{(\mu)}_2$, 20 terms are enough for the error less than $10^{-16}$ on $0\le\rho\le7.01558666981561875$
(the upper limit is the second positive zero of $J_1$).
For $\rho>14\pi$ for $M^{(\mu)}_1$ and $\rho>44.7593189976528217$ (the 14th positive zero of $J_1$),
asymptotic expansions described in Appendix \ref{sec:M1:asy} and \ref{sec:M2:asy} produce 16-digit accurate results.
For $\rho$ in the intermediate range, we divide the domain as $\cup_{n=1}^6[a_n,a_{n+1}]$
where $a_n=2\pi n$ for $M^{(\mu)}_1$ and $a_n$ is the (2n)th positive zero of $J_1$.
When $\rho\in[a_n,a_{n+1}]$ for some $n=1,\ldots,6$,
\begin{equation}
	M^{(\mu)}_m(\rho) = M^{(\mu)}_m(a_n)
	+ \frac{1}{\rho^\mu} \int_{a_n}^\rho t^{\mu-1} \begin{cases} \cos(t) \,dt & m = 1 \\ J_0(t) \,dt & m = 2 \end{cases}
\end{equation}
where $M^{(\mu)}_m(a_n)$ can be precomputed and reused, and the integrals on $[a_n,\rho\,]$ are evaluated at each time by a quadrature rule.
Since the integrands are very smooth and oscillate less than one cycle in the interval,
any high order numerical quadrature (such as Clenshaw-Curtis) with a small number of samples can compute the result up to the machine precision.

\end{enumerate}

Thus, we have presented our scheme for the construction of corrected quadrature weights
and the required formulae for the Fourier transforms of logarithmic and power-law singularities.
We conclude this section by presenting the factored form of the Helmholtz kernel in arbitrary dimension,
for which the results we have developed so far turn out to be well-suited.

\subsection{Helmholtz kernels}
\label{sec:helmholtz}
Denote by $K^k_n(r)$ the Helmholtz kernel in $\RR^n$ with complex wavenumber $k$ such that $\operatorname{Im}(k)\ge0$.
The domain of the integral (i.e. the domain of the convolution) is not necessarily $n$-dimensional.
If the convolution is performed on $\RR^m$, $n$ needs only to satisfy $n\le m+1$.
Let $z\equiv kr$ and $\nu\equiv(n-2)/2$, then the Helmholtz kernels are given by
\begin{equation}
\label{eqn:Kkn}
	K^k_n(r) = \frac{i}{4} \left( \frac{k}{2\pi r} \right)^\nu H^{(1)}_\nu(z)
	= \frac{1}{4} \left(\frac{k}{2\sqrt{\pi}}\right)^{n-2} \left( -\frac{Y_\nu(z)}{(z/2)^\nu} + i \frac{J_\nu(z)}{(z/2)^\nu} \right)
\end{equation}
where $H^{(1)}_\nu$ is the Hankel function of the first kind.
More familiar forms in the first three dimensions are
\begin{equation}
	K^k_1(r) = \frac{i e^{ikr}}{2k} ,\quad
	K^k_2(r) = \frac{i}{4} H^{(1)}_0(kr) ,\quad
	K^k_3(r) = \frac{e^{ikr}}{4\pi r} .
\end{equation}
In the limiting case with $k=0$, $K^0_n$ is simply a constant multiple of the logarithmic or the power-law singularity;
\begin{equation}
\label{eqn:Kon}
\begin{aligned}
	K^0_1(r) &= -\frac{1}{2} r ,\\
	K^0_2(r) &= -\frac{1}{2\pi} \log(r) ,\\
	K^0_n(r) &= \frac{\Gamma\left(n/2-1\right)}{4\sqrt{\pi}^n} \frac{1}{r^{n-2}} \quad\text{for $n \ge 3$} .
\end{aligned}
\end{equation}

Recall that $J_\nu(z)/(z/2)^\nu=A_n(z)/\Gamma(n/2)$ with $A_n$ defined and extensively used in previous sections.
Like $A_n$, the imaginary part of $K^k_n$ is analytic, and has the limiting value
\begin{equation}
	\lim_{r\rightarrow0} \operatorname{Im}(K^k_n(r)) = \frac{i}{4\Gamma(n/2)} \left(\frac{k}{2\sqrt{\pi}}\right)^{n-2} .
\end{equation}
Thus, the singularity is carried entirely by the real part.
In principle, for any $n$, $\operatorname{Re}(K^k_n)$ contains the same type of singularity as $K^0_n$ given above.
However, for even $n>3$, the kernel contains an additional logarithmic singularity.

\subsubsection{Odd $\bm{n}$}
\label{sec:helmholtz:odd}
Since $\nu=(n-2)/2$ is an half-integer,
we can utilize the identity, $Y_\nu(z)=(-1)^{\lceil\nu\rceil} \, J_{-\nu}(z)$, to obtain
\begin{align}
	\frac{Y_\nu(z)}{(z/2)^\nu}
	&=(-1)^{\lceil\nu\rceil} \, \frac{J_{-\nu}(z)}{(z/2)^{-\nu}} \, \frac{1}{(z/2)^{2\nu}} \\
	&= -(-1)^{\lceil n/2\rceil} \left(\frac{2}{k}\right)^{n-2} \frac{A_{4-n}(kr)}{\Gamma(2-n/2)} \frac{1}{r^{n-2}} .
\end{align}
Therefore, the factored form is given by
\begin{equation}
	K^k_n(r) = \frac{\alpha^k_n(r)}{r^{n-2}} + \w{K}^k_n(r)
\end{equation}
with smooth functions,
\begin{align}
	\alpha^k_n(r) &= \alpha^k_n(0) \, A_{4-n}(kr)
	\quad\text{where}\quad \alpha^k_n(0) = \frac{(-1)^{\lceil n/2\rceil}}{4\,\Gamma(2-n/2)\,\sqrt{\pi}^{n-2}} \\
	\w{K}^k_n(r) &= \w{K}^k_n(0) \, A_n(kr)
	\quad\;\;\;\text{where}\quad \w{K}^k_n(0) = \frac{i}{4\,\Gamma(n/2)} \left(\frac{k}{2\sqrt{\pi}}\right)^{n-2} .
\end{align}
It is convenient to use the following recurrence relations.
\begin{equation}
	\alpha^k_{n+2}(0) = \frac{n-2}{2\pi} \, \alpha^k_n(0)
	\quad\text{and}\quad
	\w{K}^k_{n+2}(0) = \frac{k^2}{2 n \pi} \, \w{K}^k_n(0) .
\end{equation}
Recall that, in the separated form of the quadrature rule (\ref{eqn:quad:sum}),
we do not need to evaluate $\w{K}^k_n$ explicitly; only the limiting value at the origin is required.
However, the values of $\alpha^k_n$ at grid points are still needed.
The following table shows explicit formulae and values for the first few dimensions.
\sep
\begin{center}
\begin{tabular}{ccccc}
\toprule
	\rule[-7pt]{0pt}{1.8eM} \dm{n} & \dm{\phi(r)} & \dm{\alpha^k_n(r)} & \dm{\alpha^k_n(0)} & \dm{\w{K}^k_n(0)} \\
\midrule
	\rule[-7pt]{0pt}{1.8eM} 1 & \dm{r} & \dm{-\sinc(kr)/2} & \dm{-1/2} & \dm{i/(2k)}  \\
	\rule[-7pt]{0pt}{1.8eM} 3 & \dm{r^{-1}} & \dm{\cos(kr)/(4\pi)} & \dm{1/(4\pi)} & \dm{ik/(4\pi)} \\
	\rule[-7pt]{0pt}{1.8eM} 5 & \dm{r^{-3}} & \dm{\big(\cos(kr)+kr\sin(kr)\big)/(8\pi^2)} & \dm{1/(8\pi^2)} & \dm{ik^3/(24\pi^2)} \\
	\rule[-7pt]{0pt}{1.8eM} 7 & \dm{r^{-5}} & \dm{\big((3-(kr)^2)\cos(kr)+3kr\sin(kr)\big)/(16\pi^3)} & \dm{3/(16\pi^3)} & \dm{ik^5/(240\pi^3)} \\
\bottomrule
\end{tabular}
\end{center}

\subsubsection{Even $\bm{n}$}
\label{sec:helmholtz:even}
The Bessel function of the second kind with an integer index $\nu=(n-2)/2$ can written (cf. \cite[p. 358]{AbSt72}) as
\begin{equation}
	Y_\nu(z)
	= -\frac{1}{\pi} \left( \frac{z}{2} \right)^{-\nu} P_\nu(z)
	+ \frac{2}{\pi} \log\left(\frac{z}{2}\right) \, J_\nu(z)
	- \frac{1}{\pi\nu!} \left(\frac{z}{2}\right)^\nu Q_\nu(z)
\end{equation}
where
\begin{align}
	P_\nu(z) &\equiv \sum_{\ell=0}^{\nu-1} \frac{(\nu-1-\ell)!}{\ell!} \left(\frac{z}{2}\right)^{2\ell} \qquad\qquad\qquad \bigg(P_0(z)\equiv0\bigg) \\
	Q_\nu(z) &\equiv \nu! \sum_{\ell=0}^\infty (-1)^\ell \, \frac{h_\ell+h_{\nu+\ell}-2\gamma_e}{\ell!\,(\nu+\ell)!} \left(\frac{z}{2}\right)^{2\ell}
\end{align}
where $h_\ell\equiv\sum_{k=1}^\ell1/k$ ($h_0\equiv0$) and $\gamma_e$ is the Euler constant.
Thus, the kernel involves two types of singularities for $n\ge4$ and can be factored as
\begin{equation}
	K^k_n(r) = \frac{\alpha^k_n(r)}{r^{n-2}} + \beta^k_n(r)\log(r) + \w{K}^k_n(r)
\end{equation}
where
\begin{align}
	\alpha^k_n(r) &= \frac{P_\nu(kr)}{4\sqrt{\pi}^n} , \\
	\beta^k_n(r) &= \beta^k_n(0) \, A_n(kr)
	\quad\text{where}\quad
	\beta^k_n(0) = -\frac{1}{2\pi\Gamma(n/2)} \left(\frac{k}{2\sqrt{\pi}}\right)^{n-2} ,
\end{align}
and
\begin{equation}
	\w{K}^k_n(r) = \frac{1}{4\Gamma(n/2)} \left(\frac{k}{2\sqrt{\pi}}\right)^{n-2}
	\bigg\{ \frac{Q_\nu(kr)}{\pi} - \frac{2}{\pi} \log\left(\frac{k}{2}\right) A_n(kr) + i\,A_n(kr) \bigg\} .
\end{equation}

The limiting values of $\alpha^k_n$ and $\beta^k_n$ are given by
\begin{align}
	\alpha^k_n(0) &= \frac{(n-4)!}{4\sqrt{\pi}^n} ,\quad \alpha^k_2(0) = 0 \\
	\beta^k_{n+2}(0) &= \frac{k^2}{2n\pi} \beta^k_{n}(0) ,\quad \beta^k_2(0) = -\frac{1}{2\pi}
\end{align}
The following table summarizes formulae and values of $\alpha^k_n$ and $\beta^k_n$ for the first few even dimensions.
\begin{center}
\begin{tabular}{ccccc}
\toprule
	\rule[-7pt]{0pt}{1.8eM} \dm{n} & \dm{\alpha^k_n(r)} & \dm{\alpha^k_n(0)} & \dm{\beta^k_n(r)} & \dm{\beta^k_n(0)} \\
\midrule
	\rule[-7pt]{0pt}{1.8eM} 2 & 0 & 0 & \dm{-J_0(kr)/(2\pi)} & \dm{-1/(2\pi)} \\
	\rule[-7pt]{0pt}{1.8eM} 4 & \dm{1/(4\pi^2)} & \dm{1/(4\pi^2)} & \dm{-k J_1(kr)/(4\pi^2 r)} & \dm{-k^2/(8\pi^2)} \\
	\rule[-7pt]{0pt}{1.8eM} 6 & \dm{\big(1+(kr/2)^2\big)/(4\pi^3)} & \dm{1/(4\pi^3)} & \dm{-k^2 J_2(kr)/(8\pi^3 r^2)} & \dm{-k^4/(64\pi^3)} \\
	\rule[-7pt]{0pt}{1.8eM} 8 & \dm{\big(2+(kr/2)^2+(kr/2)^4/2\big)/(4\pi^4)} & \dm{1/(2\pi^4)} & \dm{-k^3 J_3(kr)/(16\pi^4 r^3)} & \dm{-k^6/(768\pi^4)} \\
\bottomrule
\end{tabular}
\end{center}

The value of $\w{K}^k_n$ at the origin is given by
\begin{equation}
	\w{K}^k_n(0) = \frac{1}{4\Gamma(n/2)} \left(\frac{k}{2\sqrt{\pi}}\right)^{n-2}
	\bigg\{ \frac{h_{(n-2)/2} -2\gamma_e}{\pi} - \frac{2}{\pi}\log\left(\frac{k}{2}\right) + i \bigg\} .
\end{equation}
Hence, $\w{K}^k_n(0)$ for any even $n>2$ can be generated by the recurrence relation
\begin{equation}
	\w{K}^k_{n+2}(0) = \frac{k^2}{2n\pi} \w{K}^k_{n}(0) + \frac{1}{n^2\Gamma(n/2)\pi} \left(\frac{k}{2\sqrt{\pi}}\right)^{n}
\end{equation}
from the initial value
\begin{equation}
	\w{K}^k_2(0) = \frac{i}{4} - \frac{\gamma_e}{2\pi} - \frac{1}{2\pi} \log\left(\frac{k}{2}\right)
\end{equation}

The existence of the additional logarithmic singularity does not add any difficulty to the construction of the quadrature weights;
now, the formula (\ref{eqn:quad:sum}) simply contains one more term.
\begin{equation}
	w_\ell = \begin{cases}
	\alpha^k_n(0)\,\w{\phi}^\alpha_0 + \beta^k_n(0)\,\w{\phi}^\beta_0 + \w{K}^k_n(0) & \ell = 0 \\
	\alpha^k_n(\rl) \left(\w{\phi}^\alpha_\ell - \phi^\alpha(\rl)\right) \vphi(\rl)
	+ \beta^k_n(\rl) \left(\w{\phi}^\beta_\ell - \phi^\beta(\rl)\right) \vphi(\rl) & \ell\ne0
	\end{cases}
\end{equation}
where $\phi^\alpha(r)=1/r^{n-2}$ and $\phi^\beta(r)=\log(r)$.
The regularized singularities $\w{\phi}^\alpha_\ell$ and $\w{\phi}^\beta_\ell$ are obtained by the same procedure, independently of each other.

\section{Convergence Analysis}
In this section, we present the rate of convergence of the presented quadrature rules depending on the regularity of the data.
The main result is that our corrected trapezoidal rules applied to smooth data converge faster than any algebraic order of accuracy.
We begin with defining useful regularity classes.
The Fourier coefficients of functions in each class shares a common form of upper bounds.
Then, we present the accuracy of the (uncorrected) trapezoidal rule for the regular integral (\ref{eqn:quad:G}).
Finally, the accuracy of the corrected trapezoidal rule for the singular integral (\ref{eqn:quad:F}) is presented.

\subsection{Decay characteristics of Fourier coefficients}
In spectral contexts, the classification of functions by the decay characteristics of their spectral coefficients will always be the best one
since the error of a numerical scheme is in principle determined by the rate of decay of the coefficients.
However, in common situations, an \emph{a priori} estimate is not likely to be available.
On the contrary, the smoothness of a function is more accessible,
and upper bounds for the Fourier coefficients $|\h{F}(k)|$ can be obtained by repeated applications of integration-by-parts.
One of the earliest application of this technique (in a somewhat diffrerent direction) can be found in \cite{Ly71}.
We present the extension of one-dimensional results (cf. \cite{Bo00,Bo09}) to higher dimensions.
Although this approach has been widely exercised, the employed regularity conditions vary depending on authors
resulting, sometimes, a less tight error bound for the trapezoidal rule.
We begin with the following definition.

\begin{definition}
\label{def:I}
	For any non-negative integer $P$, we denote by $C^P_{per}$
	the class of periodic (or periodized) $C^P$ functions with $U$ as a period, which satisfy the following conditions.
	Let $F\in C^P_{per}$.
	\begin{enumerate}
	\item For $P>0$ and $P$ is even, $\Delta^{(P/2)} F \in C^0_{per}$.
	\item For $P>0$ and $P$ is odd, $\partial_j\Delta^{\lfloor P/2\rfloor} F \in C^0_{per}$ for all $j=1,\ldots,m$.
	\item For $P=0$,
	(a) $U$ consists of a finite number of disjoint sub-domains on each of which $F$ is $C^1$
	up to the sub-domain boundary and $C^2$ with $\Delta F$ in $L^1$, and
	(b) the boundary of each sub-domain is of $C^1$ and each connected component of the boundary has two adjacent sub-domains.
	\end{enumerate} 
\end{definition}

Then, an $F\in C^P_{per}$ is of the H\"older class $C^{P,1}$.
With an arbitrary assignment of the orientation, unit normal vectors are well-defined on the sub-domain boundaries,
and so is the trace of $\partial(\Delta^{\lfloor P/2\rfloor}F)/\partial n$.
The conditions on $C^0_{per}$ regarding the piecewise smoothness are to enable the application of the Green's identities.
Those piecewise smoothness enable a more precise estimate by considering
two types of manageable singularities in the (P+2)nd derivatives of a function -- $L^1$ and $H^{-1}$.
Take the function $F\equiv\max(0,1-4t^2)$ on $[-1,1]$ for example.
$F$ is in $C^0_{per}$ and its Fourier coefficients are given by $\h{F}(k)=(\sinc(\pi k/2)-\cos(\pi k/2))/(\pi k/2)^2$ ($O(|k|^{-2})$).
Less rigorously, we can expect the result without the precise evaluation;
the second derivative $F^{(2)}$ consists of two delta functions at $\pm1/2$, hence, the Fourier coefficients of $F^{(2)}$ are $O(1)$,
which results in the $O(|k|^{-2})$-decay of $\h{F}(k)$.
On the contrary, the function $\sqrt{\max(0,1-4t^2)}$ is $C^0$ but does not satisfy the piecewise regularity conditions (hence, is not in $C^0_{per}$).
Its Fourier coefficients are $J_1(\pi k/2)/(2k)$, which is of $O(|k|^{-3/2})$.
We consider the first example can be observed more frequently than the second.
The second example illustrates that we can enrich the classification by introducing (weaker) classes with half-integer indices ($L^1$ first derivatives).
However, in this paper, we follow the virtue of simplicity.

\begin{lemma}
\label{lem:CP:decay}
Let $F\in C^P_{per}$. Then,
\begin{equation}
	\left| \h{F}(k) \right| = O\left( \norm{k}^{-(P+2)} \right)
\end{equation}
\end{lemma}
\begin{proof}
(1) For $P=0$, we apply Green's second identity on each sub-domain.
Let $\Gamma$ be the union of the sub-domain boundaries with arbitrarily assigned orientation.
Utilizing $\Delta(e^{-i\pi k\cdot y})=-\pi^2\norm{k}^2\,e^{-i\pi k\cdot y}$ and
the continuity of $\partial(e^{-i\pi k\cdot y})/\partial n$, we obtain
\begin{equation}
\label{eqn:lem:CP:decay:I}
	2^m \h{F}(k) = \frac{1}{\pi^2 \norm{k}^2}
	\Bigg\{
		\int_{\Gamma} \bigg[\!\!\bigg[\frac{\partial F}{\partial n}\bigg]\!\!\bigg](y) \, e^{-i\pi k\cdot y} \, dS_y
		- \int_{U} \Delta F(y) \, e^{-i\pi k\cdot y} \, dy
	\Bigg\} ,
\end{equation}
where $[\![\cdot]\!]$ is the jump discontinuity across $\Gamma$.
Therefore, $|\h{F}(k)|=O\left(\norm{k}^{-2}\right)$.
(2) For $P=1$, we can apply Green's first identity.
By the continuity of the first derivatives, the boundary integral disappears.
\begin{equation}
	\h{F}(k) = -\frac{i}{\pi \norm{k}^2} \sum_{j=1}^m k_j \int_{U} \partial_j F(y) \, e^{-i\pi k\cdot y} \, dy .
\end{equation}
Since $\partial_j F\in C^0_{per}$ by definition, the integral for each $j$ is $O\left(\norm{k}^{-2}\right)$.
By the H\"older inequality, $|\sum k_j|\le\sqrt{m}\norm{k}$.
(3) For $P>1$, apply Green's second identity $\lfloor P/2\rfloor$ times with the boundary integrals vanishing. Then,
\begin{equation}
	\h{F}(k) = \left( \frac{-1}{\pi^2 \norm{k}^2} \right)^{\lfloor P/2\rfloor}
	\int_{U} \Delta^{\lfloor P/2\rfloor} F(y) \, e^{-i\pi k\cdot y} \, dy .
\end{equation}
Apply the result of (1) if $P$ is even, and apply (2) if odd.
\end{proof}

The upper bound given by Lemma $\ref{lem:CP:decay}$ is somewhat conservative;
for $m>1$, the integrals in (\ref{eqn:lem:CP:decay:I}) contribute to an additional decay (of possibly fractional order).
However, in one-dimension, the integral on $\Gamma$ becomes point-wise evaluations of the jump discontinuity,
hence, we cannot expect a faster decay.
Also, one should note that the smoothness is not the only factor which governs the decay characteristics.
Consider the function $\max(0,1-4r^2)^2$ which is in $C^1_{per}$.
Its Fourier coefficients are $((12-4(\pi k/2)^2)\sinc(\pi k/2)-12\cos(\pi k/2))/(\pi k/2)^4$,
hence, are of $O(|k|^{-4})$, which is faster than the estimated $O(|k|^{-3})$.
Although they are in different regularity classes, the function $\max(0,1-4r^2)^2\in C^1_{per}$
exhibits the same rate of convergence by the trapezoidal rule as a more smooth function $\max(0,1-4r^2)^3\in C^2_{per}$.

\begin{lemma}
\label{lem:CP:aconv}
Let $F\in C^P_{per}$. Then, $\sum_{k\in\Zm}|\h{F}(k)|<\infty$ if $P>m-2$.
\end{lemma}
\begin{proof}
Rewrite the infinite series by $|\h{F}(0)|+\sum_{n=1}^\infty\sum_{\norm{k}^2=n} |\h{F}(k)|$.
Let $r(n)$ be the number of lattice points on the sphere $\big\{\norm{k}^2=n\big\}$, which is also known as the sum-of-squares function.
Then, $R(n)\equiv\sum_{n=0}^N r(n)$ is the number of lattice points in the ball $\big\{\norm{k}^2\le n\big\}$,
and the Gauss' circle problem in $\Zm$ states that $R(n)=O(n^{m/2})$ (cf. \cite{MiSz47}).
Since $|\h{F}(k)|=O(n^{-(P+2)/2})$ for $\norm{k}^2=n$ by Lemma \ref{lem:CP:decay}, the series is bounded above by
$\h{F}(0) + C \sum_{n=1}^\infty r(n)/n^{(P+2)/2}$ for some $C$.
Since $r(n)/n^{(P+2)/2}=(R(n)-R(n-1))/n^{(P+2)/2}=O((n^{m/2}-(n-1)^{m/2})/n^{(P+2)/2})$
and $n^{m/2}-(n-1)^{m/2}=O(n^{m/2-1})$, $r(n)/n^{(P+2)/2}=O(n^{-(P-m+4)/2})$,
which implies that the series converges if $P>m-2$.
\end{proof}

\begin{corollary}
\label{cor:CP:uconv}
Let $F\in C^P_{per}$. The multiple Fourier series
\begin{equation}
\label{eqn:fseries}
	\sum_{k\in\Zm} \h{F}(k) e^{i\pi k\cdot y}
\end{equation}
converges uniformly to $F(y)$ if $P>m-2$.
\end{corollary}

The criterion $P>m-2$ is somewhat excessive in higher dimensions.
As mentioned earlier, if we consider the additional decay, the convergence criterion can be relaxed.

\subsection{Accuracy of the trapezoidal rule}
By substituting the Fourier series (\ref{eqn:fseries}) into the definition of DFT (\ref{eqn:dft}), we obtain
\begin{equation}
	\h{F}_k = \frac{1}{2^m \Nb} \sum_{k'\in\Zm} \h{F}(k') \sum_{\ell\in\IN} e^{i\pi(k'-k)\cdot\yl} .
\end{equation}
The inner sum is $2^m \Nb$ if $k'_j=k_j+2q_jN_j$ for all $j$ and for some $q\in\Zm$,
and vanishes otherwise.
Therefore,
\begin{align}
\label{eqn:aliasing}
	\h{F}_k &= \sum_{q\in\Zm} \h{F}(k_1+2q_1N_1,\ldots,k_m+2q_mN_m) \nonumber \\
	&= \h{F}(k) + \sum_{q\in\Zm\setminus\{0\}} \h{F}(k_1+2q_1N_1,\ldots,k_m+2q_mN_m) .
\end{align}
The above equation describes that the coefficients $\h{F}_k$ obtained from the DFT
contain the aliasing error from the harmonics ($\sum_{q\ne0}$).

Notice that $\h{F}_0$ is the integral of $F$ on $U$ obtained by the trapezoidal rule
(for periodic functions, hence, considering only one side of the boundary),
and $\h{F}(0)$ is the exact integral.
Hence, an error bound of the trapezoidal rule is given by
\begin{equation}
	| \h{F}(0) - \h{F}_0 | \le
	\sum_{q\in\Zm\setminus\{0\}} | \h{F}(2q_1N_1,\ldots,2q_mN_m) | .
\end{equation}
Let $N_{min}=\min(N_1,\ldots,N_m)$. The following theorem states the order of accuracy of the trapezoidal rule applied to a $C^P_{per}$ function.

\begin{theorem}[The trapezoidal rule]
\label{thm:regular}
For any $F\in C^P_{per}$ with $P>m-2$. The error of the trapezoidal rule applied to $F$ is $O(N_{min}^{-(P+2)})$.
\end{theorem}
\begin{proof}
Let $p=(2q_1N_1,\ldots,2q_mN_m)$. Then, $\norm{p}\ge 2N_{min}\norm{q}$ and
$$|\h{F}(p)| \le C/(N_{min}\norm{q})^{(P+2)}$$ for some $C$ by Lemma \ref{lem:CP:decay}.
Rewriting $\sum_{q\in\Zm\setminus\{0\}}$ by $\sum_{n=1}^\infty\sum_{\norm{q}^2=n}$,
\begin{equation*}
	| \h{F}(0) - \h{F}_0 | \le \frac{C}{N_{min}^{P+2}} \sum_{n=1}^\infty \frac{r(n)}{n^{(P+2)/2}}
\end{equation*}
where $r(n)$ is defined in the proof of Lemma \ref{lem:CP:aconv} and is $O(n^{m/2-1})$.
The series on the right side converges if $P>m-2$.
\end{proof}

\subsection{Accuracy of the singular trapezoidal rule}
From (\ref{eqn:fseries}) and (\ref{eqn:aliasing}), the error of the interpolation $\w{F}$ defined by (\ref{eqn:intp}) can be written as
\begin{equation}
\begin{split}
	F(y)-\w{F}(y) &= \sum_{k\not\in\IN} \h{F}(k)\,e^{i\pi k\cdot y} \\
	&- \sum_{k\in\IN}\sum_{q\in\Zm\setminus\{0\}} \h{F}(k_1+2q_1N_1,\ldots,k_m+2q_mN_m)\,e^{i\pi k\cdot y} .
\end{split}
\end{equation}
Hence, the error of the singular quadrature rule (\ref{eqn:quad:F}) is given by
\begin{equation}
\label{eqn:error:singular}
\begin{split}
	&\left| \int_U (F(y)-\w{F}(y))\,\phi(r)\,dy \right| \le \sum_{k\not\in\IN} |\h{F}(k)| \, |\h{\phi}(k)| \\
	&+ \sum_{k\in\IN} |\h{\phi}(k)| \sum_{q\in\Zm\setminus\{0\}} |\h{F}(k_1+2q_1N_1,\ldots,k_m+2q_mN_m)|
\end{split}
\end{equation}

\begin{theorem}[The singular quadrature rule]
\label{thm:singular}
Let $F\in C^P_{per}$ with $P>m-2$.
The error of the quadrature rule (\ref{eqn:quad:F}) is $O(N_{min}^{-(P+2-(m-1)/2)})$.
\end{theorem}
\begin{proof}
(1) First, consider the first summation in (\ref{eqn:error:singular}).
Since $\h{F}(k)=O(\norm{k}^{-(P+2)})$ by Lemma \ref{lem:CP:decay}
and $\h{\phi}(k)=O(\norm{k}^{-(m+1)/2})$ by Lemma \ref{lem:phi},
\begin{equation*}
	\sum_{k\not\in\IN} |\h{F}(k)| |\h{\phi}(k)| \le \sum_{n=N_{min}^2}^\infty\sum_{\norm{k}^2=n} \frac{C}{n^{(P+m/2+5/2)/2}} \\
	= \sum_{n=N_{min}^2}^\infty \frac{C\,r(n)}{n^{(P+m/2+5/2)/2}}
\end{equation*}
where $r(n)=O(n^{m/2-1})$ as in the proof of Lemma \ref{lem:CP:aconv}.
Thus, the right side is bounded by $C\sum_{n=N_{min}^2}^\infty n^{-(P-m/2+9/2)/2}$,
which converges if $P>(m-5)/2$ and is $O(N_{min}^{-(P+2-(m-1)/2)})$.
(2) Consider the second term in (\ref{eqn:error:singular}).
Since $|k_j|\le N_j$ for $j=1,\ldots,m$, the minimum of the parabola $(k_j+2q_jN_j)^2$ is $\min((2q_j\pm1)^2N_j^2)$ if $q_j\ne0$.
Hence, for $q_j\ne0$, $(k_j+2q_jN_j)^2\ge(2|q_j|-1)^2N_j^2\ge q_j^2N_j^2$, and the last inequality holds for $q_j=0$ also.
Therefore, $\norm{(k_1+2q_1N_1,\ldots,k_m+2q_mN_m)}\ge\norm{q}N_{min}$
and $|\h{F}(k_1+2q_1N_1,\ldots,k_m+2q_mN_m)|=O((N_{min}\norm{q})^{-(P+2)})$ by Lemma \ref{lem:CP:decay}.
Then, the second term is bounded by
\begin{equation}
\label{eqn:thm:singular:b}
	C \bigg\{ |\h{\phi}(0)| + \sum_{n=1}^{N_{min}^2}\sum_{\norm{k}^2=n} n^{-(m+1)/4} \bigg\}
	\bigg\{ N_{min}^{-(P+2)} \sum_{p=1}^\infty \sum_{\norm{q}^2=p} p^{-(P+2)/2} \bigg\} ,
\end{equation}
where we utilized Lemma \ref{lem:phi} to get the upper bound for $|\h{\phi}(k)|$.
The infinite series in the second pair of braces converges if $P>m-2$.
Since
\begin{equation}
	\sum_{n=1}^{N_{min}^2}\sum_{\norm{k}^2=n} n^{-(m+1)/4} \le C' \sum_{n=1}^{N_{min}^2} n^{(m-5)/4} = O(N_{min}^{(m-1)/2}) ,
\end{equation}
the second term in (\ref{eqn:error:singular}) is $O(N_{min}^{-(P+2-(m-1)/2)})$.
Thus, (1) and (2) complete the proof.
When the kernel is less singular and $|\h{\phi}(k)|$ decays faster than $O(\norm{k}^{-(m+1)/2})$,
(\ref{eqn:thm:singular:b}) becomes dominant and determines the rate of convergence.
If $|\h{\phi}(k)|$ decays faster than $O(\norm{k}^{-m})$, the first term in (\ref{eqn:thm:singular:b}) is $O(1)$ and
the rate of convergence is the same as the regular trapezoidal rule, i.e. $(P+2)$.
\end{proof}

Combined with Theorem \ref{thm:regular}, Theorem \ref{thm:singular} indicates that
the order of accuracy of the corrected trapezoidal rule presented in this paper is at least $(P+2-(m-1)/2)$ for the integral (\ref{eqn:Ix}) with $f\in C^P_{per}$.
If $f$ is smooth, the corrected trapezoidal rules converges faster than any algebraic order.
The dimension-dependent degradation from the intended $(P+2)$ (as the regular trapezoidal rule) by the amount of $-(m-1)/2$
is more like a technical outcome; in actual experiments with $m\le3$, we have not experienced any obvious degradation in the rate of convergence
and results strongly imply that the corrected trapezoidal rule exhibits the same rate of convergence as the usual trapezoidal rule without the singular kernel.
The same is true for the condition $P>m-2$.
Our cautious conjecture is that, as briefly mentioned previously, there is an additional decay for $m>1$ originating from the integrals in (\ref{eqn:lem:CP:decay:I}),
which cancels (at least, a part of) the degradation.

The intention of the somewhat complicated proof of Theorem \ref{thm:singular} can be well illustrated by considering simple one-dimensional cases.
For $C^P_{per}$ functions, both the trapezoidal rule and the corrected trapezoidal rule exhibit $(P+2)$nd order convergence.
The interpolation error, however, is only of $(P+1)$st order.
The mutual cancellation of overshooting and undershooting, which has been well explained for the regular trapezoidal rule,
holds identically for our singular quadrature rules.
In the singular cases, the additional one in the rate of convergence originates from the decay characteristics of $|\h{\phi}(k)|$.

\section{Numerical Examples}

\subsection{Test of convergence: Helmholtz kernels with $\bm{k=0}$}
\label{sec:conv:I}
We evaluate the convolution with the non-oscillatory kernels on $\RR^m$ ($m=1,2,3$):
\begin{equation}
	u = K^0_n * f
\end{equation}
for $n=m$ and $m+1$.
For $n=m$, $u$ corresponds to the volume potential induced by the surce $f$.
For $n=m+1$, the convolution corresponds to the application of the single layer operator
as a boundary to boundary integral operator on $m$-dimensional flat boundary in $\RR^{m+1}$.
Convolutions were performed on uniform grids with the spacing $6/N$ on the physical domain $[-3,3]^m$.
The errors are reported in $L_\infty$ norm from the values of the numerical and the exact solutions obtained at grid points.

Table \ref{tab:p:a} and \ref{tab:p:b} summarize the results of one-dimensional convolution with the logarithmic kernel: $K^0_2=-\log(r)/(2\pi)$.
Reference solutions are computed up to the machine precision using adaptive quadratures.
For the demonstration, we selected three different sources,
\begin{equation}
\begin{aligned}
	f_G(r) &= \exp\left(-(r/a)^2\right) &&(a=1/2) \\
	f_B(r) &= \exp\left(12-12/(1-(r/a)^2)\right) &&(a=2) \\
	f_P(r) &= \max\left(0,1-(r/a)^2\right)^7 &&(a=2) ,
\end{aligned}
\end{equation}
based on their regularity characteristics.
The Gaussian ($f_B$) is analytic but is not compactly supported. 
However, on the boundary of the domain employed, $|f_G|\sim10^{-16}$ and any significant error of the domain truncation has not been observed.
Actually, the spectrum of the Gaussian (in any dimension) exhibits the most rapid decay, which makes it a perfect specimen for the test of the spectral accuracy.
Table \ref{tab:p:a} shows a clear super-algebraic convergence of our quadrature rule; each refinement doubles the number of correct digits.
The bump function $f_B$ is smooth and compactly supported, hence, satisfies our assumptions faithfully.
However, its Fourier coefficients decays quite slowly compared with the Gaussian.
The left column of Table \ref{tab:p:b} shows that the estimated order keeps increasing but the increase is slower than that of the Gaussian cases in Table \ref{tab:p:a}.
The function $f_P$ is not smooth but in $C^6_{per}$.
Hence, the estimated order is fixed to the algebraic order of 8 (right column in Table \ref{tab:p:b}),
which is expected in the convergence analysis.

Recall \S\ref{sec:therule} that we can use arbitrarily refined grid during the construction of the quadrature weights.
Such a refined construction increases only the construction time but does not affect the convolution time
since the obtained weights will be truncated in the frequency domain.
The objective of the use of a refined construction grid is to suppress errors originating from insufficient samplings of $\alpha$, $\vphi$, and $K$.
Since $K^0_n$ are non-oscillatory ($\alpha$ is a constant and $\w{K}=0$), the error during the construction originates mostly from $\vphi$.
In Table \ref{tab:p:a}, the results on the left column are obtained using the weights constructed on the same grid as the data.
For the results on the right column, we doubled the sampling frequency of the construction grid.
For small $N$, the difference is negligible since the error from the data is dominant.
However, the interpolation error of the Gaussian decreases more rapidly than that of $\vphi$.
Hence, if we use the same grid, the construction error stand out eventually.
The remedy is simple; we can use a slightly higher sampling frequency for the construction.
As for $f_B$ and $f_P$, their interpolation errors decrease more slowly than $\vphi$,
and the results are almost indistinguishable with or without using a refined construction grid (see Table~\ref{tab:p:b}).

Table~\ref{tab:p:c} shows the results for $(K^0_n*f_G)$ in higher dimensions ($m=2,3$ and $n=m,m+1$).
For the Gaussian, we know the exact solutions, which are given by
\begin{equation}
\begin{aligned}
	(m=2,n=2)\quad
	u(r) &= \frac{a^2}{4} \log\left( \frac{\exp(-E_1(\rho^2))}{\rho^2} \right) - \frac{a^2}{2} \log(a)
	\\
	(m=2,n=3)\quad
	u(r) &= \frac{a\sqrt{\pi}}{4} \, \exp\left(-\frac{\rho^2}{2}\right) I_0\left(\frac{\rho^2}{2}\right)
	\\
	(m=3,n=3)\quad
	u(r) &= \frac{a^2\sqrt{\pi}}{4} \frac{\erf(\rho)}{\rho}
	\\
	(m=3,n=4)\quad
	u(r) &= \frac{a}{2\sqrt{\pi}} \exp\left(-\frac{\rho^2}{2}\right)
		\int_0^1 \exp\left(-\frac{\rho^2}{2} t^2 \right) I_0\left( \frac{\rho^2}{2} (1-t^2) \right) \,dt
\end{aligned}
\end{equation}
where $\rho=r/a$ and $a=1/2$.
As in the one-dimensional case, we can observe clear spectral rates of convergences.

\begin{table}[t]
\begin{center}
\begin{tabular}{r r@{.}l c r@{.}l c}
\toprule
	& \multicolumn{3}{c}{$f_G$ (without ref.)}
	& \multicolumn{3}{c}{$f_G$ (with ref.)} \\
\cmidrule(r){2-4}
\cmidrule(r){5-7}
	$N$
	& \multicolumn{2}{c}{$E_N$} & $\log_2\frac{E_{N/2}}{E_N}$
	& \multicolumn{2}{c}{$E_N$} & $\log_2\frac{E_{N/2}}{E_N}$ \\
\midrule
	  5 & $5$ & $58\times10^{-2}$   &            --- & $5$ & $59\times10^{-2}$   &     --- \\
	10 & $3$ & $26\times10^{-3}$   & $\;\:4.1$ & $3$ & $26\times10^{-3}$   & $\;\:4.1$ \\
	20 & $1$ & $30\times10^{-6}$   &   $11.3$ & $1$ & $30\times10^{-6}$   & $11.3$ \\
	40 & $3$ & $32\times10^{-13}$ &   $21.9$ & $3$ & $89\times10^{-16}$ & $31.6$ \\
\bottomrule\\
\end{tabular}
\end{center}
\caption{$L_\infty$ errors of $(K^0_2*f_G)$ on $\RR$:
(left) without and (right) with using the doubly-refined grid for the construction of the weights.}
\label{tab:p:a}
\end{table}

\begin{table}[t]
\begin{center}
\begin{tabular}{r r@{.}l c r@{.}l c}
\toprule
	& \multicolumn{3}{c}{$f_B$}
	& \multicolumn{3}{c}{$f_P$} \\
\cmidrule(r){2-4}
\cmidrule(r){5-7}
	$N$
	& \multicolumn{2}{c}{$E_N$} & $\log_2\frac{E_{N/2}}{E_N}$
	& \multicolumn{2}{c}{$E_N$} & $\log_2\frac{E_{N/2}}{E_N}$ \\
\midrule
	  5 & $4$ & $17\times10^{-2}$   &           --- & $1$ & $46\times10^{-2}$   &       --- \\
	10 & $7$ & $21\times10^{-4}$   & $\;\:5.8$ & $5$ & $65\times10^{-5}$   & $8.0$ \\
	20 & $1$ & $45\times10^{-6}$   & $\;\:9.0$ & $2$ & $36\times10^{-7}$   & $7.9$ \\
	40 & $9$ & $25\times10^{-10}$ &  $10.6$ & $7$ & $31\times10^{-10}$ & $8.3$ \\
	80 & $2$ & $36\times10^{-14}$ &  $15.3$ & $4$ & $33\times10^{-12}$ & $7.4$ \\
\bottomrule\\
\end{tabular}
\end{center}
\caption{$L_\infty$ errors of $(K^0_2*f_B)$ and $(K^0_2*f_P)$ on $\RR$.}
\label{tab:p:b}
\end{table}

\begin{table}[t]
\begin{center}
\begin{tabular}{r r@{.}l c r@{.}l c}
\toprule
	& \multicolumn{3}{c}{$(m=2,n=2)$}
	& \multicolumn{3}{c}{$(m=2,n=3)$} \\
\cmidrule(r){2-4}
\cmidrule(r){5-7}
	$N$
	& \multicolumn{2}{c}{$E_N$} & $\log_2\frac{E_{N/2}}{E_N}$
	& \multicolumn{2}{c}{$E_N$} & $\log_2\frac{E_{N/2}}{E_N}$ \\
\midrule
	  5 & $1$ & $06\times10^{-1}$   &           ---  & $4$ & $88\times10^{-2}$    &            --- \\
	10 & $3$ & $96\times10^{-3}$   & $\;\:4.7$ & $4$ & $70\times10^{-3}$    & $\;\:3.4$ \\
	20 & $8$ & $99\times10^{-7}$   &   $12.1$ & $2$ & $35\times10^{-6}$    &  $11.0$ \\
	40 & $5$ & $55\times10^{-16}$ &   $30.6$ & $3$ & $33\times10^{-16}$ &   $32.7$ \\
\bottomrule\\
\end{tabular}
\begin{tabular}{r r@{.}l c r@{.}l c}
\toprule
	& \multicolumn{3}{c}{$(m=3,n=3)$}
	& \multicolumn{3}{c}{$(m=3,n=4)$} \\
\cmidrule(r){2-4}
\cmidrule(r){5-7}
	$N$
	& \multicolumn{2}{c}{$E_N$} & $\log_2\frac{E_{N/2}}{E_N}$
	& \multicolumn{2}{c}{$E_N$} & $\log_2\frac{E_{N/2}}{E_N}$ \\
\midrule
	  5 & $4$ & $04\times10^{-2}$   &            --- & $1$ & $41\times10^{-2}$    &          --- \\
	10 & $4$ & $10\times10^{-3}$   & $\;\:3.3$ & $5$ & $03\times10^{-3}$    & $\;\:1.5$ \\
	20 & $1$ & $19\times10^{-6}$   &   $11.8$ & $3$ & $22\times10^{-6}$    &   $10.6$ \\
	40 & $1$ & $05\times10^{-15}$ &   $30.1$ & $3$ & $05\times10^{-16}$ &   $33.3$ \\
\bottomrule\\
\end{tabular}
\end{center}
\caption{$L_\infty$ errors of $(K^0_n*f_G)$ on $\RR^m$.}
\label{tab:p:c}
\end{table}

\subsection{Test of convergence: Helmholtz kernels with $\bm{k\ne0}$}
\label{sec:conv:II}
Convergence tests were performed for the Helmholtz kernels $K^k_n$ with nonzero wavenumber.
First, we present results with real wavenumber $k=2\pi$ on $[-3,3]^m$.
The Gaussian $f_G$ was selected as the source.
Since the Fourier coefficients of $f_G$ decay rapidly, the quadrature rule, if is is properly constructed, should exhibit a similar fast convergence.
Unlike non-oscillatory kernels of the previous section, we do not know the exact solution even for the Gaussian.
Moreover, it is very difficult and expensive with adaptive quadratures to obtain the reference solutions accurate up to the machine-precision.
Hence, the reference solutions were evaluated only at the origin, where the error is likely to be largest.
The results are summarized in Table \ref{tab:h:a} and \ref{tab:h:b},
from both of which we can observe the expected spectral rates of convergences.

\begin{table}[htb]
\begin{center}
\begin{tabular}{r r@{.}l c r@{.}l c}
\toprule
	& \multicolumn{3}{c}{$(m=1,n=2,\text{without ref.})$}
	& \multicolumn{3}{c}{$(m=1,n=2,\text{with ref.})$} \\
\cmidrule(r){2-4}
\cmidrule(r){5-7}
	$N$
	& \multicolumn{2}{c}{$E_N$} & $\log_2\frac{E_{N/2}}{E_N}$
	& \multicolumn{2}{c}{$E_N$} & $\log_2\frac{E_{N/2}}{E_N}$ \\
\midrule
	  5 & $4$ & $11\times10^{-2}$   &            --- & $7$ & $95\times10^{-2}$   &     --- \\
	10 & $4$ & $66\times10^{-2}$   & $\!\!-0.2$ & $6$ & $47\times10^{-3}$   & $\;\:3.6$ \\
	20 & $2$ & $89\times10^{-4}$   &  $\;\:7.3$ & $2$ & $82\times10^{-6}$  & $11.2$ \\
	40 & $2$ & $61\times10^{-11}$ &   $23.4$ & $3$ & $93\times10^{-17}$ & $36.1$ \\
\bottomrule\\
\end{tabular}
\end{center}
\caption{$L_\infty$ errors of $(K^k_2*f_G)$ on $\RR$ with $k=2\pi$:
(left) without and (right) with using the doubly-refined grid for the construction of the weights.}
\label{tab:h:a}
\end{table}

\begin{table}[htb]
\begin{center}
\begin{tabular}{r r@{.}l c r@{.}l c}
\toprule
	& \multicolumn{3}{c}{$(m=2,n=2)$}
	& \multicolumn{3}{c}{$(m=2,n=3)$} \\
\cmidrule(r){2-4}
\cmidrule(r){5-7}
	$N$
	& \multicolumn{2}{c}{$E_N$} & $\log_2\frac{E_{N/2}}{E_N}$
	& \multicolumn{2}{c}{$E_N$} & $\log_2\frac{E_{N/2}}{E_N}$ \\
\midrule
	  5 & $4$ & $01\times10^{-2}$   &           ---  & $1$ & $02\times10^{-1}$    &            --- \\
	10 & $1$ & $14\times10^{-2}$   & $\;\:1.8$ & $1$ & $26\times10^{-2}$    & $\;\:3.0$ \\
	20 & $2$ & $46\times10^{-6}$   &   $12.2$ & $4$ & $77\times10^{-6}$    &  $11.4$ \\
	40 & $2$ & $08\times10^{-17}$ &   $36.8$ & $2$ & $55\times10^{-16}$ &   $34.1$ \\
\bottomrule\\
\end{tabular}
\begin{tabular}{r r@{.}l c r@{.}l c}
\toprule
	& \multicolumn{3}{c}{$(m=3,n=3)$}
	& \multicolumn{3}{c}{$(m=3,n=4)$} \\
\cmidrule(r){2-4}
\cmidrule(r){5-7}
	$N$
	& \multicolumn{2}{c}{$E_N$} & $\log_2\frac{E_{N/2}}{E_N}$
	& \multicolumn{2}{c}{$E_N$} & $\log_2\frac{E_{N/2}}{E_N}$ \\
\midrule
	  5 & $4$ & $61\times10^{-2}$   &            --- & $1$ & $13\times10^{-1}$    &          --- \\
	10 & $1$ & $52\times10^{-2}$   & $\;\:1.6$ & $1$ & $81\times10^{-2}$    & $\;\:2.6$ \\
	20 & $2$ & $95\times10^{-6}$   &   $12.3$ & $6$ & $17\times10^{-6}$    &   $11.5$ \\
	40 & $2$ & $96\times10^{-17}$ &   $36.5$ & $4$ & $13\times10^{-16}$ &   $33.8$ \\
\bottomrule\\
\end{tabular}
\end{center}
\caption{$L_\infty$ errors of $(K^k_n*f_G)$ on $\RR^m$ with $k=2\pi$.}
\label{tab:h:b}
\end{table}

Similarly to Table \ref{tab:p:a}, Table \ref{tab:h:a} illustrates the enhancement in the accuracy by using a higher sampling frequency for the construction grid.
However, the results with the oscillating kernels exhibit quite different aspects:
first, the level of the construction error is significantly larger than the previous non-oscillatory cases.
Second, the corruption occurs also for small $N$.
Recall, for non-oscillatory kernels in the previous section, the main source of the construction error was $\vphi$, whose Fourier coefficients decay quite rapidly.
Although the interpolation error of the Gaussian exhibits faster decay than that of $\vphi$, the Gaussian is a somewhat special case.
For many other functions (such as $f_B$ and $f_P$), $\vphi$ is good enough not to cause such a issue.
However, for oscillatory kernels, the major source of the construction error is $\alpha$ and $\w{K}$.
Simply, we need a construction grid fine enough to represent the oscillating $\alpha$ with similar accuracy as the data
($\w{K}$ oscillates similarly to $\alpha$).
Hence, the sampling frequency of the construction grid should be increased proportionally to the wavenumber.

On the other hand, if the sampling frequency of the data is not high enough for the representation of the oscillating kernel,
the sampling frequency is not high enough for the representation of the solution either.
Hence, ironically, in most of practical applications such as examples in \S\ref{sec:ls} and \S\ref{sec:bie},
we cannot expect the enhanced accuracy by merely increasing the sampling frequency of the construction grid.
The first convolution from the exactly given data may be accurate at each grid point.
However, the interpolation error of the first solution will not be as accurate as that of the data, due to the insufficient grid resolution.
This first solution with insufficient accuracy will be (a part of ) the source of the next convolution.
Hence, if an appropriate grid is chosen for the data and the solution, we can use the same grid for the construction.

\begin{figure}[t]
\begin{center}
\includegraphics[width=0.48\textwidth]{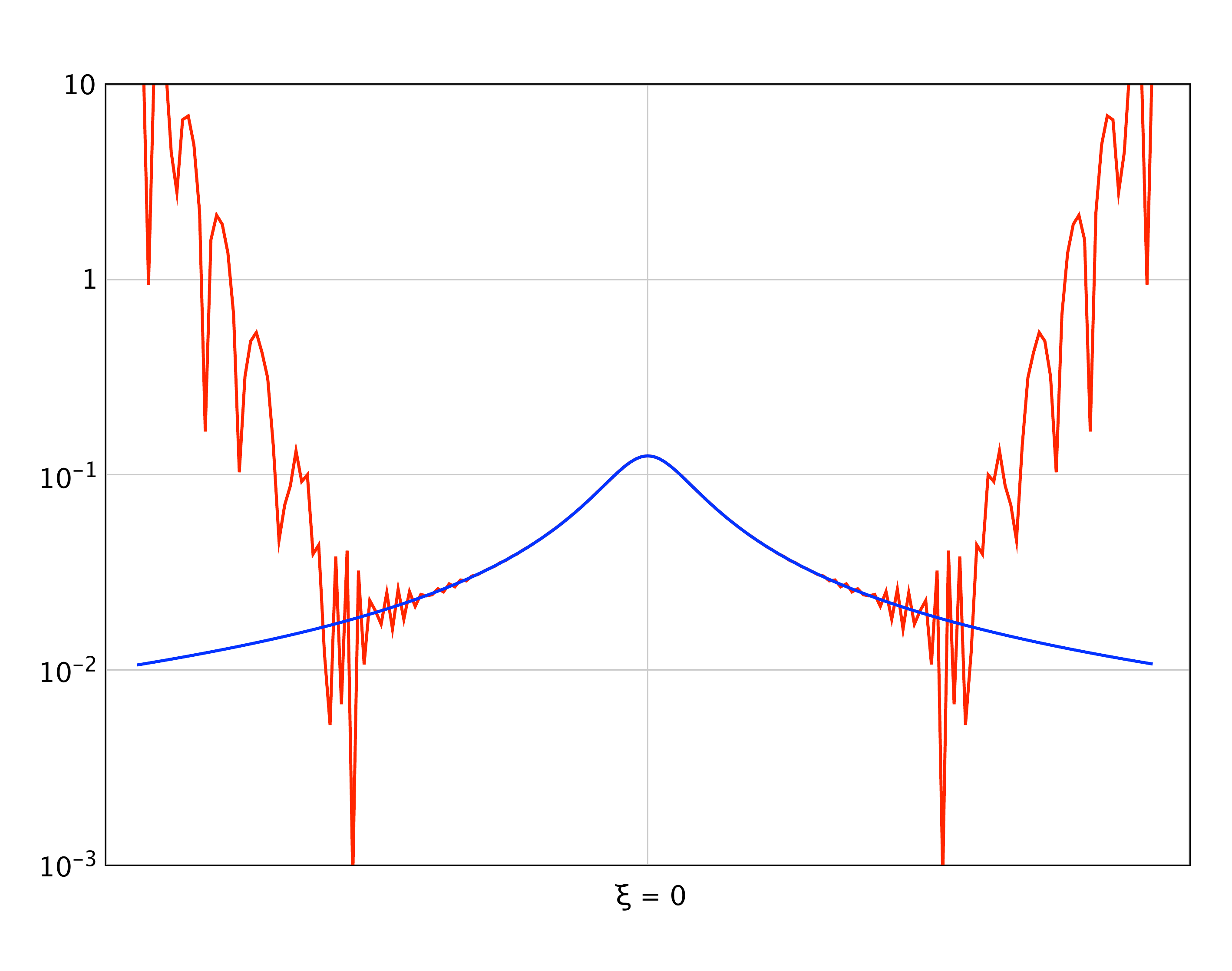}
\includegraphics[width=0.48\textwidth]{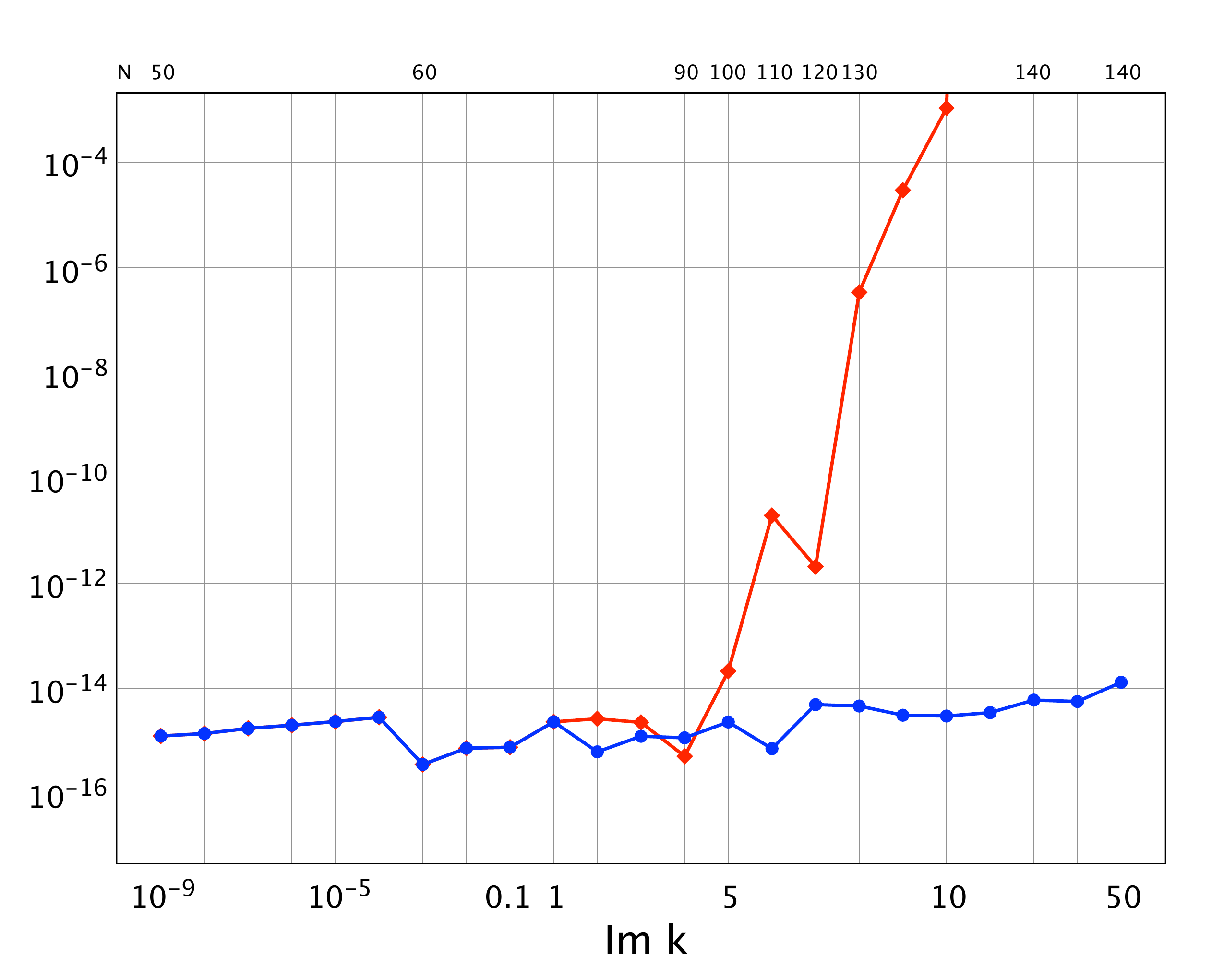}
\end{center}
\caption{The Helmholtz kernel $K^k_2$ on $\RR$: (left) the weights for $k=4i$ in the frequency domain and (right) errors up to $k=50i$.}
\label{fig:h:a}
\end{figure}

The situation becomes more complicated if $\operatorname{Im}(k)\ne0$.
Take $K^k_2$ with a pure imaginary wavenumber $k=i\lambda$ for example.
\begin{equation}
	K^k_2(r) = \frac{i}{4} H^{(1)}_0(i \lambda r) = \frac{K_0(\lambda r)}{2\pi} \sim \sqrt{\frac{1}{8\pi\lambda r}}e^{-\lambda r}
\end{equation}
As described in \S\ref{sec:helmholtz}, the kernel can be factored as $K^k_2(r)=\beta^k_2(r)\log(r)+\w{K^k_2}(r)$ with
\begin{equation}
	\beta^k_2(r)=-\frac{J_0(i \lambda r)}{2\pi}=-\frac{I_0(\lambda r)}{2\pi} \sim \sqrt{\frac{1}{8\pi^3\lambda r}} \left( e^{\lambda r} - i e^{-\lambda r} \right) .
\end{equation}
$I_0$ and $K_0$ are the modified Bessel functions of the first and the second kinds.
Notice that the exponentially decaying kernel is decomposed by two exponentially increasing $\beta^k_2$ and $\w{K}^k_2$.
With finite precision arithmetic, the meaningful signal ($\sim e^{-\lambda r}/\sqrt{r}$) in $\beta^k_2$ is completely lost for $\lambda r \gg1$.
Thus, we should choose sufficiently small $\BR$ so that the $\beta^k_2$ evaluated within $\BR$ remains small $\sim1$.
The left figure in Figure~\ref{fig:h:a} shows the spectrums of two sets of the weights generated for $\lambda=4$.
For the accurate weights, the radius of $\BR$ is set to $1$.
The inaccurate weights obtained with $\BR$ of the radius 3 deviate wildly after the mid-frequency.
As we increase $\lambda$ with the fixed radius of 3, the points of deviation move to lower frequencies, and soon only the noise remains.
Since the construction grid should be fine enough to resolve $\vphi$ with $\BR$, its sampling frequency should be increased accordingly.
In order to prevent the generation of a huge construction grid due to the high sampling frequency, we can take only a small subset of grid containing $\BR$.
Outside of the subset, the kernel is practically zero,
hence, the spectrum of the weights can be obtained by zero-padded DFT followed by the truncation of the unnecessary high frequency results.

In the right figure of Figure~$\ref{fig:h:a}$, the errors with the appropriately constructed weights are less than $10^{-14}$.
Without the $\lambda$-dependent control on $\BR$, the error rapidly increases after $\lambda\ge5$.

\subsection{Application: Lippmann-Schwinger equation}
\label{sec:ls}
Consider the acoustic scattering problem in an inhomogeneous medium.
Let the constant $k$ be the wavenumber of the medium at infinity (or we may call it the ambient wavenumber) and
let $n$ the non-constant refractive index such that $(n-1)$ is compactly supported.
Then, for time-harmonic problems, the acoustic pressure $u$ satisfies the equation,
\begin{equation}
	\Delta u + k^2 n(x) \,u = 0 .
\end{equation}
Let $u^i$ be the given incident wave which satisfies $\Delta u^i + k^2\,u^i = 0$.
We assume the scattered field $u^s=u-u^i$ satisfies the radiation condition at infinity.
The explicit form of the radiation condition depends on the dimension.
In this paper, we consider an example in $\RR^2$, where the radiating scattered field satisfies
\begin{equation}
	\lim_{r\rightarrow\infty} \sqrt{r} \left( \frac{\partial u^s}{\partial r} - i\,k\,u \right) = 0 .
\end{equation}

The above scattering problem is equivalent to solving the integral equation,
\begin{equation}
\label{eqn:ls}
	u(x) - k^2 \int_D K_2^k(r) \left( n(y)-1 \right) u(y) \,dy = u^i(x) ,
\end{equation}
which is also known as the Lippmann-Schwinger equation.
The linear operator on the left side of the equation (denoted by $\mathcal{L}$) can be written as
\begin{equation}
	\mathcal{L} = I - k^2\,\mathcal{K}\,\mathcal{N}
\end{equation}
in terms of the convolution $\mathcal{K}u=K_2^k * u$ and the multiplication $\mathcal{N}u=(n-1)u$.
Since $\mathcal{K}$ is compact and $\mathcal{N}$ is bounded,
$\mathcal{L}$ is a Fredholm operator of the second kind and the problem is well-posed (cf. \cite[\S8]{CoKr98}).

In this example, we consider a medium with three \emph{bumps},
\begin{equation}
	n(y) = 1 - 0.9 \sum_{i=1}^3 e^{2 \left( 1- \left( 1 - (y_1-a_i)^2-(y_2-b_i)^2 \right)^{-1} \right) } ,
\end{equation}
with the centers at $(a_i,b_i) = (1,0)$, $(-1,3)$, and $(-1,-3)$.
Inside of the bumps, the wavespeed slows down.
The equation $\mathcal{L}u=u^i$ is solved using GMRES for incident planewave with the wavenumber $k=5\pi$ and the direction $(1,0)$.
The domain of the computation is $[-6,6]^2$.
In order to measure the error, the reference solution is computed on the uniform $1280\times1280$ grid.
For the convergence test presented in Table~\ref{tab:ls}, we measured $L_\infty$ error on the grid points.
Note that, unlike the convergence tests in previous sections, the reported errors are not just for the convolution itself but
includes the combined effect of the whole solution procedure.
Figure~\ref{fig:ls} depicts the total field the scattering problem.

\begin{table}[t]
\begin{center}
\begin{tabular}{r r@{.}l c}
\toprule
	$N$ & \multicolumn{2}{c}{$E_N$} & $\log_2\frac{E_{N/2}}{E_N}$ \\
\midrule
	  80 & $1$ & $42\times10^{-1}$   &         --- \\
	160 & $2$ & $08\times10^{-4}$   & $\;\:9.4$ \\
	320 & $2$ & $07\times10^{-7}$   & $10.0$ \\
	640 & $7$ & $42\times10^{-11}$ & $11.4$ \\
\bottomrule\\
\end{tabular}
\end{center}
\caption{$L_\infty$ error of the solution of (\ref{eqn:ls}): $k=5\pi$ on $[-6,6]^2$.}
\label{tab:ls}
\end{table}

\begin{figure}[t]
\begin{center}
\includegraphics[clip,width=0.7\textwidth]{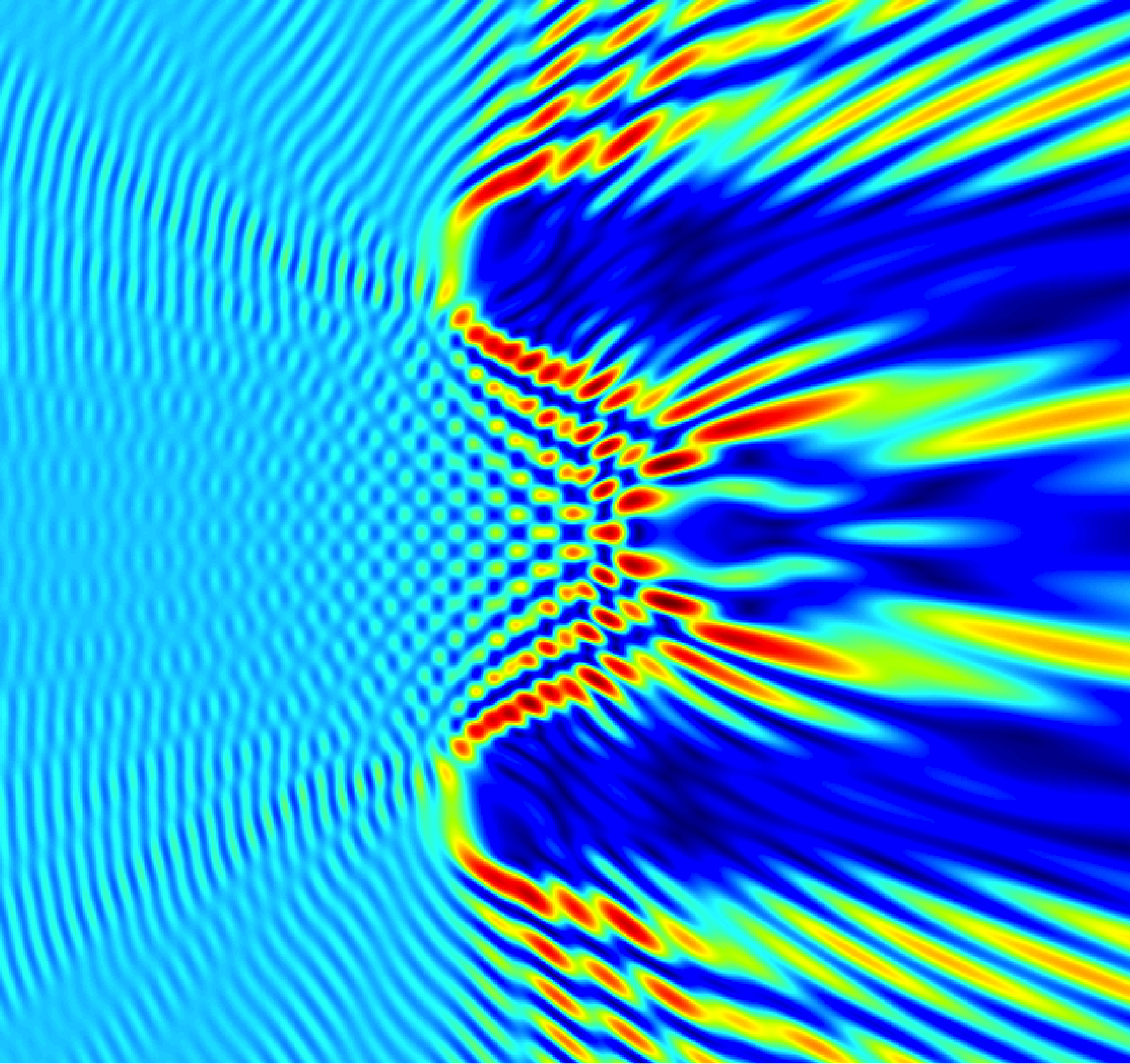}
\end{center}
\caption{Solution of the Lippmann-Schwinger equation (\ref{eqn:ls}): $k=5\pi$ on $[-6,6]^2$.}
\label{fig:ls}
\end{figure}

\subsection{Application: boundary integral equation}
\label{sec:bie}
The presented quadrature rules can be applied for integral operators on curves in $\RR^m$.
Let each curve be represented by a smooth periodic parameterization: $s\in[0,L]\mapsto y(s)\in\RR^m$.
Assume $y'(s)\ne0$ for every $s\in[0,L]$.
For $s,t\in[0,L]$, define the ratio,
\begin{equation}
	\gamma(s,t) = \frac{r(s,t)}{|t-s|} \quad\text{where}\quad r(s,t) = \norm{y(t)-y(s)} .
\end{equation}
Then, $\gamma(t,t)=\norm{y'(t)}$ and $\gamma(s,\cdot)$ is smooth, positive, and periodic on $[0,L]$.
The logarithmic and the power-law singularities can be rewritten as
\begin{align}
	\log(r) &= \log(|s-t|) + \log(\gamma) \\
	r^{-\nu} &= \gamma^{-\nu} |s-t|^{-\nu} .
\end{align}
Thus, the above radial singularities in $\RR^m$ can be recast to singularities of the same kind on $[0,L]$.

In this example, we solve the exterior scattering problem,
\begin{equation}
\label{eqn:bie}
	\frac{1}{2}\,\psi(s) + \int_\Gamma D(s,t)\,\psi(t) \,\norm{y'(t)}\,dt - i k \int_\Gamma S(s,t)\,\psi(t) \,\norm{y'(t)}\,dt = -u^i(s) ,
\end{equation}
where the combine integral equation is employed to avoid the interior resonance issue.
The single layer kernel $S$ is identical to the Helmholtz kernel $K^k_2$ except the additional $\log(\gamma)$ in the smooth remainder;
\begin{align}
	S(s,t) &\equiv \frac{i}{4} H^{(1)}_0(kr) \equiv \alpha_s(s,t)\log(|s-t|) + \w{S}(s,t) \\
	\alpha_s(s,t) &= -\frac{J_0(kr)}{2\pi} \\
	\alpha_s(t,t) &= -\frac{1}{2\pi} \\
	\w{S}(t,t) &= \frac{i}{4} - \frac{\gamma_{euler}}{2\pi} - \frac{1}{2\pi}\log\left(\frac{k\norm{y'(t)}}{2}\right)
\end{align}
where $\gamma_{euler}$ is the Euler constant.
The double layer kernel $D$ is given by
\begin{align}
	D(s,t) &\equiv n(t)\cdot\nabla K^k_2(r)
	= \frac{d}{dr} \left( \frac{i}{4} H^{(1)}_0(kr) \right) \left\{ \frac{n(t)\cdot(y(t)-y(s))}{r} \right\} \nonumber \\
	&= \frac{kr}{4} \bigg\{ Y_1(kr) - i J_1(kr) \bigg\} \bigg\{ \frac{n(t)\cdot(y(t)-y(s))}{r^2} \bigg\} ,
\end{align}
where $n(t)$ is the outward normal vector at $t$.
Since (cf. \S\ref{sec:helmholtz:even})
\begin{equation}
	z\,Y_1(z) = \frac{2 z J_1(z)}{\pi} \log(z) - \frac{2}{\pi} + O(z^2) ,
\end{equation}
the kernel can be factored by
\begin{align}
	D(s,t) &\equiv \alpha_d(s,t)\log(|s-t|)+\w{D}(s,t) \\
	\alpha_d(s,t) &= \frac{kr J_1(kr)}{2\pi} \bigg\{ \frac{n(t)\cdot(y(t)-y(s))}{r^2} \bigg\} \\
	\alpha_d(t,t) & = 0 \\
	\w{D}(t,t) &= -\frac{1}{2\pi} \lim_{s\to t} \bigg\{ \frac{n(t)\cdot(y(t)-y(s))}{r^2} \bigg\} = \frac{c(t)}{2\pi\norm{y'(t)}^2}
\end{align}
where $c(t)$ is the curvature at $t$.
The above information is all we need to construct the quadrature weights (at each target point).
Beware that a fast convolution cannot be applied for this case since the kernels are not functions of $(t-s)$,
hence, the resulting discrete operator is not a circular matrix.

Figure~\ref{fig:bie} shows the total field constructed from the obtained solution $\psi$.
Each scatterer is a translation of the popular \emph{kite} shape,
\begin{equation}
	y(t) = \left( \cos t + 0.65\cos(2t) - 0.65 , 1.5\sin t \right) \quad t \in [0,2\pi] .
\end{equation}
The incident planewave with $k=5\pi$ comes from the upper-left corner of the domain $[-8,8]^2$ to the lower-right corner.
For the interaction between disconnected curves, the kernels become smooth periodic, hence, the usual trapezoidal rule can be used.

Table~\ref{tab:bie} shows the error of the solution $\psi$ on the boundary and the error of the far-field evaluate from the obtained $\psi$.
The far-field signature can be computed by
\begin{equation}
	u_\infty(\h{x}) = \frac{e^{-i\pi/4}}{\sqrt{8\pi k}}
	\int_\Gamma \left( k\,n(t)\cdot\h{x} + k \right) e^{-ik\,\h{x}\cdot y(t)} \,\psi(t) \,\norm{y'(t)}\,dt
\end{equation}
where $\h{x}=x/|x|$, $x\in\RR^2$.
The reference solution is obtained with $N=640$.

\begin{table}[t]
\begin{center}
\begin{tabular}{r r@{.}l c r@{.}l c}
\toprule
	& \multicolumn{3}{c}{solution error}
	& \multicolumn{3}{c}{far-field error} \\
\cmidrule(r){2-4}
\cmidrule(r){5-7}
	$N$
	& \multicolumn{2}{c}{$E_N$} & $\log_2\frac{E_{N/2}}{E_N}$
	& \multicolumn{2}{c}{$E_N$} & $\log_2\frac{E_{N/2}}{E_N}$ \\
\midrule
	  80 & $1$ & $16\times10^{-1}$   &         ---  & $1$ & $45\times10^{-1}$   &        --- \\
	160 & $1$ & $61\times10^{-5}$   & $12.8$ & $2$ & $06\times10^{-7}$   & $19.4$\\
	320 & $4$ & $17\times10^{-14}$ & $28.5$ & $1$ & $34\times10^{-14}$ & $23.9$\\
\bottomrule\\
\end{tabular}
\end{center}
\caption{$L_\infty$ errors of the exterior scattering problem (\ref{eqn:bie}): $k=5\pi$.}
\label{tab:bie}
\end{table}

\begin{figure}[t]
\begin{center}
\includegraphics[width=0.7\textwidth]{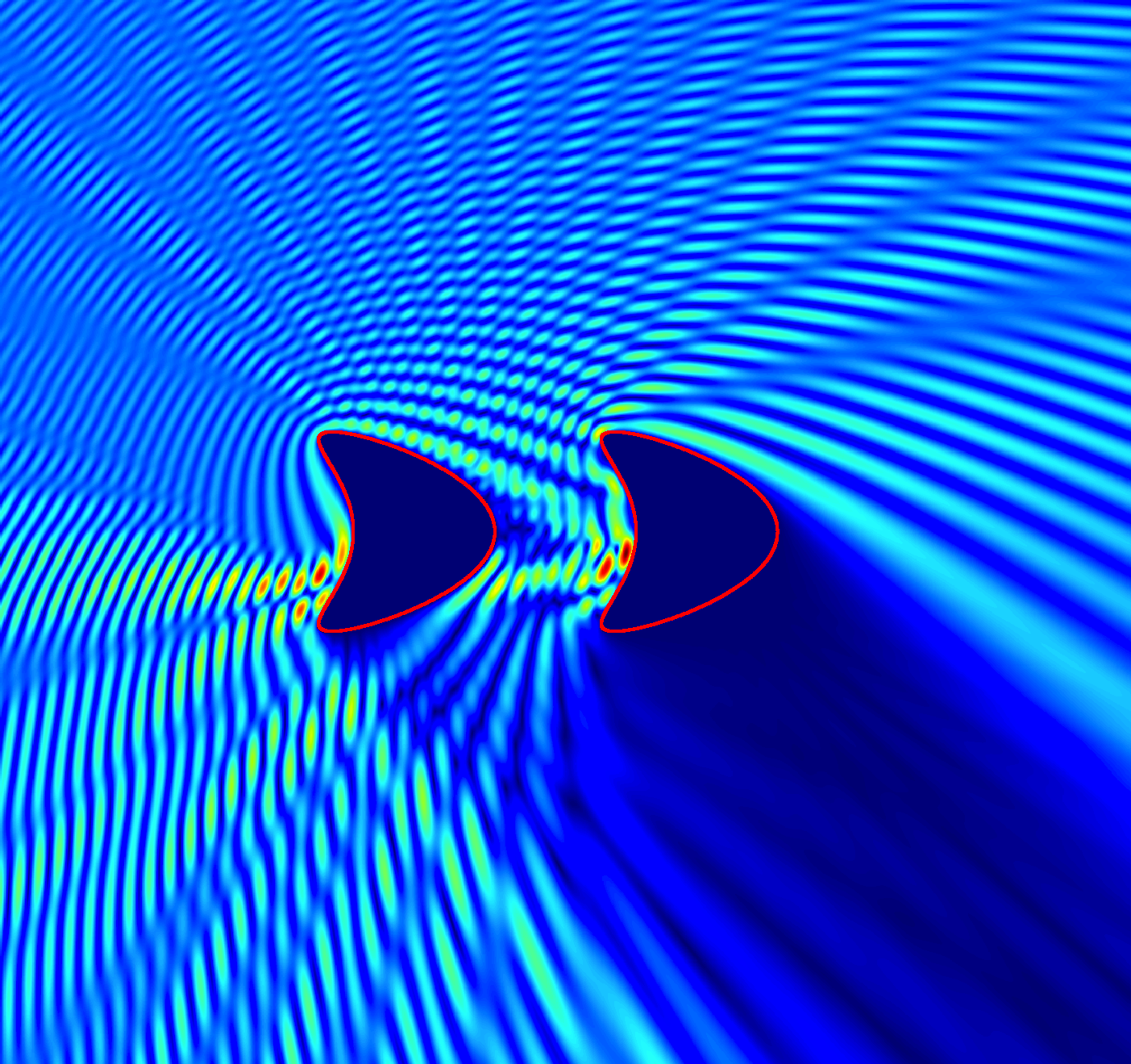}
\end{center}
\caption{Total field of the exterior scattering problem (\ref{eqn:bie}) on $[-8,8]^2$: $k=5\pi$.
The obstacles are centered at $\pm2$ on $x$ axis.}
\label{fig:bie}
\end{figure}

\subsection{Conclusions}
We have presented the construction scheme for corrected trapezoidal rules for integral operators with weakly singular kernels.
Numerical results show that the quadrature rules converge as fast as the trapezoidal rule applied for non-singular periodic integrands
as predicted by the presented convergence analysis.
Especially, for smooth data, the quadrature rules exhibit super-algebraic convergence.
The examples in \S\ref{sec:ls} and \S\ref{sec:bie} demonstrate the effectiveness of the proposed quadratures in the high precision solve of integral equations.


\appendix

\section{Asymptotic Expansion of $M^{(\mu)}_1$}\label{sec:M1:asy}
Dividing the domain of integral, we rewrite $M^{(\mu)}_1(\rho)$ as
\begin{equation}
	M^{(\mu)}_1(\rho) = M^{(\mu)}_1(a) + I_a(\mu,\rho) \quad\text{where}\quad I_a(\mu,\rho) \equiv \frac{1}{\rho^\mu} \int_a^\rho t^{\mu-1} \cos(t) \,dt .
\end{equation}
Applying integration-by-parts $(2N)$ times, we obtain
\begin{equation}
	I_a(\mu,\rho) = I_a^N(\mu,\rho) + \frac{(-1)^N}{\rho^\mu} \prod_{\ell=1}^{2N} (\ell-\mu) \int_0^\rho t^{\mu-1-2N} \cos(t)\,dt ,
\end{equation}
where
\begin{equation}
\begin{split}
	I_a^N(\mu,\rho)
	&= \rho^{-1}\sin(\rho)\,P^N(\mu,\rho) - a^{-1}(a/\rho)^\mu\sin(a)\,P^N(\mu,a) \\
	&+ \rho^{-2}\cos(\rho)\,Q^N(\mu,\rho) - a^{-2}(a/\rho)^\mu\cos(a)\,Q^N(\mu,a) .
\end{split}
\end{equation}
Polynomials $P^N$ and $Q^N$ are defined by
\begin{align}
	P^N(\mu,t) &= \sum_{\ell=0}^{N-1} C_\ell \, t^{-2\ell} \quad\text{where}\quad C_\ell = (-1)^{\ell} \prod_{j=1}^{2\ell} (j-\mu) \\
	Q^N(\mu,t) &= (\mu-1) \, P^N(\mu-1,t) .
\end{align}
Then, the relative error is given by
\begin{equation}
	\left| I_a(\mu,\rho) - I_a^N(\mu,\rho) \right| < \frac{(2N)!}{a^{2N}} \left| I_a(\rho) \right| .
\end{equation}
By choosing $a=14\pi$ and $N=13$, the relative error is less than $10^{-16}$.
Moreover, with $a$ being an integer multiple of $\pi$, $I^N_a$ has a slightly simpler form.

\section{Asymptotic Expansion of $M^{(\mu)}_2$}\label{sec:M2:asy}
Similarly to $M^{(\mu)}_1$, we divide the domain of integral;
\begin{equation}
	M^{(\mu)}_2(\rho) = M^{(\mu)}_2(a) + I_a(\mu,\rho) \quad\text{where}\quad I_a(\mu,\rho) \equiv \frac{1}{\rho^\mu} \int_a^\rho t^{\mu-1} J_0(t) \,dt
\end{equation}
Utilizing $t\,J_0(t)=(t\,J_1(t))'$ and $J_1(t)=-J'_0(t)$, $(2N)$ applications of the integration-by-parts result in
\begin{equation}
	I_a(\mu,\rho) = I_a^N(\mu,\rho) + \frac{(-1)^N}{\rho^\mu} \prod_{\ell=1}^{N} (2\ell-\mu)^2 \int_0^\rho t^{\mu-1-2N} J_0(t)\,dt 
\end{equation}
where
\begin{equation}
\begin{split}
	I_a^N(\mu,\rho)
	&= \rho^{-1}J_1(\rho)\,P^N(\mu,\rho) - a^{-1}(a/\rho)^\mu J_1(a)\,P^N(\mu,a) \\
	&+ \rho^{-2}J_0(\rho)\,Q^N(\mu,\rho) - a^{-2}(a/\rho)^\mu J_0(a)\,Q^N(\mu,a)
\end{split}
\end{equation}
with
\begin{align}
	P^N(\mu,t) &= \sum_{\ell=0}^{N-1} C_\ell \, t^{-2\ell} \quad\text{where}\quad C_\ell = (-1)^{\ell} \prod_{j=1}^{\ell} (2j-\mu)^2 \\
	Q^N(\mu,t) &= \sum_{\ell=0}^{N-1} (\mu-2\ell-2) \, C_\ell \, t^{-2\ell} .
\end{align}
The relative error of the asymptotic expansion is given by
\begin{equation}
	\left| I_a(\mu,\rho) - I_a^N(\mu,\rho) \right| < \frac{4^N(N!)^2}{a^{2N}} \left| I_a(\rho) \right| .
\end{equation}
With $a=44.7593189976528217$ (which is the 14th zero of $J_1$) and $N=15$, the relative error is less than $10^{-16}$.

\bibliographystyle{amsplain}
 

\end{document}